\documentclass[11pt]{amsart}
\usepackage[top=1.2in,bottom=1.2in,left=1.2in,right=1.2in,marginpar=1in]{geometry}
\usepackage{amsfonts}
\usepackage{amsmath}
\usepackage{amsthm}
\usepackage{url}
\usepackage[all]{xy}
\usepackage{labelfig}
\usepackage{latexsym}
\usepackage{amssymb}
\usepackage[cp850]{inputenc}
\usepackage{graphicx}
\usepackage{dsfont}
\usepackage{soul}
\usepackage{stackrel}
\usepackage[linktocpage=true]{hyperref}
\usepackage{graphicx}
\usepackage{caption}
\usepackage{subcaption}
\usepackage{lipsum}
\usepackage{xfrac}
\usepackage{setspace}
\usepackage{bigints}
\usepackage{ stmaryrd } 
\usepackage{calrsfs}
\usepackage{comment}
\usepackage{relsize}
\usepackage{array}
\usepackage{xcolor}
\usepackage[normalem]{ulem}
\usepackage{hyperref}

\usepackage[autostyle, english = american]{csquotes}
\MakeOuterQuote{"}

\usepackage[dvipsnames]{xcolor}

\usepackage[cal=euler]{mathalfa}
\hypersetup{
    colorlinks = true,
    linkcolor = {NavyBlue},
    citecolor = {NavyBlue},
}
\usepackage{stmaryrd} 
\usepackage{mathtools} 
\usepackage{braket}
\usepackage{tikz-cd}

\let\oldtocsection=\tocsection

\let\oldtocsubsection=\tocsubsection

\let\oldtocsubsubsection=\tocsubsubsection

\renewcommand{\tocsection}[2]{\hspace{0em}\oldtocsection{#1}{#2}}
\renewcommand{\tocsubsection}[2]{\hspace{1em}\oldtocsubsection{#1}{#2}}
\renewcommand{\tocsubsubsection}[2]{\hspace{2em}\oldtocsubsubsection{#1}{#2}}

\renewcommand{\hat}{\widehat}

\newcounter{Lcomments}
\newcommand{\ludo}[1]{\textbf{\color{Green}(L\arabic{Lcomments})} \marginpar{\scriptsize\raggedright\textbf{\color{Green}(L\arabic{Lcomments})Ludo: }#1}
	\addtocounter{Lcomments}{1}}

 \newcounter{Jcomments}

 \newcounter{Todo}

\newtheorem{sat}{Theorem}[section]		
\newtheorem{lem}[sat]{Lemma}
\newtheorem{kor}[sat]{Corollary}			
\newtheorem{prop}[sat]{Proposition}
\newtheorem{bei}{Example}
\newtheorem{defi}{Definition}
\newtheorem*{defi*}{Definition}			
\newtheorem*{bei*}{Example}
\newtheorem*{sat*}{Theorem}				
\newtheorem*{kor*}{Corollary}
\newtheorem*{rmk*}{Remark}				
\newtheorem{quest}{Question}	
\newtheorem*{quest*}{Question}

\let\ssection=\section
\renewcommand{\section}{\setcounter{equation}{0}\ssection}

\newtheorem*{namedtheorem}{\theoremname}
\newcommand{\theoremname}{testing}
\newenvironment{named}[1]{\renewcommand{\theoremname}{#1}\begin{namedtheorem}}{\end{namedtheorem}}

\theoremstyle{remark}
\newtheorem*{bem}{Remark}

\newtheorem*{namedtheoremr}{\theoremnamer}
\newcommand{\theoremnamer}{testing}

\newcommand{\BA}{\mathbb A}

			\newcommand{\BH}{\mathbb H}

			\newcommand{\BN}{\mathbb N}
			
			\newcommand{\BR}{\mathbb R}
\newcommand{\BS}{\mathbb S}

			\newcommand{\BZ}{\mathbb Z}

\newcommand{\CA}{\mathcal A}			
			\newcommand{\calD}{\mathcal D}
\newcommand{\CE}{\mathcal E}			
\newcommand{\CG}{\mathcal G}			
\newcommand{\CI}{\mathcal I}			\newcommand{\CJ}{\mathcal J}

\newcommand{\CO}{\mathcal O}			\newcommand{\CP}{\mathcal P}
			\newcommand{\CR}{\mathcal R}
			\newcommand{\CT}{\mathcal T}
			\newcommand{\CV}{\mathcal V}
\newcommand{\CW}{\mathcal W}		\newcommand{\CX}{\mathcal X}
\newcommand{\CY}{\mathcal Y}			\newcommand{\CZ}{\mathcal Z}

\newcommand{\FM}{\mathfrak m}

\newcommand{\FF}{\mathfrak f}

\usepackage{bm} 

\newcommand{\D}{\partial}

\newcommand{\bs}{\backslash}

\DeclareMathOperator{\PSL}{PSL}		
\DeclareMathOperator{\Id}{Id}		
\DeclareMathOperator{\Isom}{Isom}	

\DeclareMathOperator{\Map}{Map}

\DeclareMathOperator{\Ker}{Ker}

\DeclareMathOperator{\area}{area}

\DeclareMathOperator{\Leb}{Leb}

\DeclareMathOperator{\Liouville}{Liouv}
\DeclareMathOperator{\kinematic}{kin}

\newcommand{\fsubd}{\mathrel{{\scriptstyle\searrow}\kern-1ex^d\kern0.5ex}}
\newcommand{\bsubd}{\mathrel{{\scriptstyle\swarrow}\kern-1.6ex^d\kern0.8ex}}

\renewcommand{\le}{\leqslant}
\renewcommand{\ge}{\geqslant}
\renewcommand{\emptyset}{\varnothing}

\title{Can you hear the shape of a hyperbolic surface? Now for real.}
\author{Ludovico Battista}
\address{Max-Planck-Institut f\"ur Mathematik, Bonn}
\email{ludox73@gmail.com}
\author{Juan Souto}
\address{IRMAR, Universit\'e de Rennes, Rennes}
\email{jsoutoc@proton.me}

\date{April 2026}

\begin{document}

\begin{abstract} 
We associate a musical instrument, a {\em hyperbolic marimba}, to every pair $(X,\Gamma)$ where $X$ is a hyperbolic surface and $\Gamma\subset X$ a simple multicurve labeled with musical keys. It works as follows: take a geodesic and every time it hits $\Gamma$, play the corresponding note. In this paper we investigate to which extent the so-produced melodies characterize $(X,\Gamma)$ up to isometry. 

In the accompanying website \href{https://ludox73.github.io/HyperMarimba/story.html}{{\em HyperMarimba}} \cite{webpage}, the reader can actually listen to the produced melodies. They can also visualize some of the phenomena we investigate.
\end{abstract}

\maketitle

\section{Introduction}
In 1966 Kac wrote the magnifically titled paper {\em Can one hear the shape of a drum?}, asking whether a plane domain was determined, up to isometry, by the Dirichlet eigenvalues of the Laplacian, the frequencies at which a drum of that shape would vibrate \cite{Kac}. Years later, isospectral domains were built by Gordon, Webb and Wolpert \cite{Gordon-Webb-Wolpert1,Gordon-Webb-Wolpert2}, answering Kac's question in the negative. Flat isospectral tori--of dimension 16--had been constructed even before Kac's paper by Milnor \cite{Milnor} and closed isospectral hyperbolic surfaces were built later on by Vign\'eras \cite{Vigneras} and Sunada \cite{Sunada}. For closed hyperbolic surfaces, the spectrum of the Laplacian determines and is determined by the (unmarked) length spectrum, that is the set of numbers, counted with multiplicity, which arise as the length of some closed geodesic. This is why questions around the extent to which the length spectrum and its cousins determine a hyperbolic surface are often packaged under the "can one hear?" type of problem. This has historic value, but, if you allow, it is also a bit of a marketing gimmick: since there is no (smooth) isometric embedding of a closed hyperbolic surface in $\BR^3$, it would be really hard to build an actual drum that sounded like the surface in question. In this paper we investigate if one can hear for real the shape of a hyperbolic surface.

Our approach is the following. Suppose that $X$ is a closed connected hyperbolic surface of genus $g\ge 2$ and $\Gamma$ a simple multicurve in $X$, that is a collection of disjoint simple closed geodesics. Label each component of $\Gamma$ with a musical note--different notes for different components--and think of the pair $(X,\Gamma)$ as a {\em hyperbolic marimba}, a musical instrument whose sound provides a sharp start to each note. Now, given a unit speed geodesic $\rho_v^X:\BR\to X$ we get a {\em melody} as follows: if at time $t\in\BR$ the point $\rho_v^X(t)$ lies on $\eta\in\Gamma$, play the key associated to $\eta$. The interested reader can go to the website \href{https://ludox73.github.io/HyperMarimba/story.html}{{\em HyperMarimba}} \cite{webpage} and actually listen to the melodies of some randomly chosen geodesics in genus $2$ hyperbolic marimbas $(X,\Gamma)$. Here "random" means that their initial tangent vectors $v=\dot\rho^X_v(0)\in T^1X$ are chosen at random with respect to the Liouville probability measure. In Section \ref{sec ergodic} we will give a more precise meaning to the word "random" and the "melody" of a random vector is defined in Section \ref{sec definition of melodies}, but an intuitive understanding of what we mean suffices for now.

\begin{figure}[h]
\begin{center}
\includegraphics[width=0.6\textwidth]{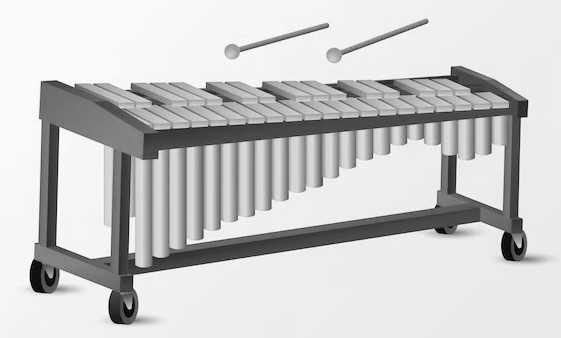}
\end{center}
\caption{A classical marimba}
\end{figure}

Although different geodesics in the marimba $(X,\Gamma)$ will produce different melodies, what we study is to which extent random melodies allow us to distinguish between different marimbas, very much as one would likely recognize the different origins of randomly chosen Finns, Indians, Texans, Peruvians, or Iranians playing their traditional music. More specifically, we want to figure out if for random vectors $v\in T^1X$ and $v'\in T^1X'$ one can hear from the melodies of the geodesics $\rho_v^X$ and $\rho_{v'}^{X'}$ whether two hyperbolic marimbas
$(X,\Gamma)$ and $(X',\Gamma')$ are {\em isometric} or not, that is whether there is an isometry $\phi:X\to X'$ sending, in a label preserving way, $\Gamma$ to $\Gamma'$. The scientifically minded reader can file this under the rubric of inverse problems.

Let us also make clear what we mean by "hearing". When humans listen to language or music, what they are doing is recognizing patterns in what they are physically hearing. An example of such a pattern is the set of notes one hears. For random vectors $v\in T^1X$, the geodesic $\rho_v^X$ meets every component of the multicurve $\Gamma$, hence every note is eventually played. It follows that we can distinguish marimbas $(X,\Gamma)$ and $(X',\Gamma')$ which use different sets of notes to label the components of $\Gamma$ and $\Gamma$. To avoid trivial cases we suppose from now on that
\begin{quote}
    when we are comparing two hyperbolic marimbas $(X,\Gamma)$ and $(X',\Gamma')$ the same notes have been used to label the components of $\Gamma$ and of $\Gamma'$.
\end{quote}
As we will see, besides the keys being played, one can hear much more from random melodies. To make this precise we need to define what we understand under a {\em motif}.

\begin{defi}
Let $(X,\Gamma)$ be a hyperbolic marimba. A {\em motif} of length $k\ge 0$ is a tuple $[(\eta_i)_{i=0,\dots,k},(t_i)_{i=1,\dots,k},\epsilon]$ consisting of a finite sequence of notes $\eta_i\in\Gamma$, a finite sequence of times $0<t_1<\dots< t_k$, and some $\epsilon>0$. 

The motif $[(\eta_i)_{i=0,\dots,k},(t_i)_{i=1,\dots,k},\epsilon]$ {\em is played at time} $t$ in the melody of a geodesic $\rho^X_v:\BR\to X$ if there are $0=s_0<s_1<s_2<\dots<s_k$ with $s_i\in[t_i,t_i+\epsilon]$ for $i=1,\dots,k$ such that the following holds:
\begin{itemize}
    \item $\rho_v^X(t+s_i)\in\eta_i$ for $i=0,\dots,k$, and
    \item $\{t,t+s_1\dots,t+s_k\}=\{s\in[t,t+s_k]\text{ with }\rho_v^X(s)\in\Gamma\}$.
\end{itemize}
\end{defi}

\begin{bem}
    The positive number $\epsilon>0$ in $[(\eta_i)_{i=0,\dots,k},(t_i)_{i=1,\dots,k},\epsilon]$ is there to give to a randomly chosen geodesic a chance to play the motif. If even the most accomplished performers diverge slightly from the timing indicated in their scores, what is to be expected from random ones?
\end{bem}

Basic ergodic theory implies that for every motif and every random $v\in T^1X$ the limit 
\begin{equation}\label{eq frequency exists}
\FF^{(X,\Gamma)}_{[(\eta_i),(t_i),\epsilon]}=\lim_{T\to\infty}\frac 1T\#\left\{t\in[0,T]\middle\vert\begin{array}{l}
\text{ the motif }[(\eta_i),(t_i),\epsilon]\text{ is played}\\ \text{ at time }t\text{ in the melody of }\rho_v^X\end{array}\right\}
\end{equation}
exists and is independent of the random vector $v$--see Proposition \ref{prop fequencies tell it all}. In other words, in the melody of a random geodesic $\rho_v^X$, every motif appears with a certain frequency (possibly $0$) which is independent of $v$. We refer to $\FF^{(X,\Gamma)}_{[(\eta_i),(t_i),\epsilon]}$ as the {\em frequency} of the motif $[(\eta_i),(t_i),\epsilon]$.

\begin{defi}
    Two hyperbolic marimbas $(X,\Gamma)$ and $(X',\Gamma')$ for which we have a label preserving bijection $\Gamma\simeq\Gamma'$ are {\em isomelodic} if we have 
    $$\FF^{(X,\Gamma)}_{[(\eta_i),(t_i),\epsilon]}=\FF^{(X',\Gamma')}_{[(\eta_i),(t_i),\epsilon]}$$
    for every motif $[(\eta_i),(t_i),\epsilon]$.
\end{defi}

Our first result is that isomelodic marimbas are essentially indistinguishable:

\begin{sat}\label{sat measure theory}
    Two hyperbolic marimbas $(X,\Gamma)$ and $(X',\Gamma')$ are isomelodic if and only if there is a full measure set $U\subset T^1X$ and an absolutely continuous map
    $$\Phi:U\to T^1X'$$
    such that for all $v\in U$ the $\Gamma$-melody of the geodesic in $X$ corresponding to $v$ is exactly the same one that the $\Gamma'$-melody of the geodesic in $X'$ corresponding to $\Phi(v)$: exactly the same notes being played at exactly the same times.
\end{sat}

Recall that a measurable map between two measure spaces is absolutely continuous if the push-forward of the measure on the domain is absolutely continuous with respect to the measure in the target.

Theorem \ref{sat measure theory} would be totally pointless if the map $\Phi$ were always induced by an isometry of the surface. This is not the case. 

\begin{bei}\label{example separating pairs}
    There are isomelodic hyperbolic marimbas $(X,\Gamma)$ and $(X,\Gamma')$ where $X,X'$ are non-isometric surfaces of genus $2$, and both $\Gamma\subset X$ and $\Gamma'\subset X'$ are non-separating pairs.
\end{bei}

\begin{bei}\label{example separating and non-separating}
    There are isomelodic hyperbolic marimbas $(X,\Gamma)$ and $(X,\Gamma')$ where $X,X'$ are non-isometric surfaces of genus $2$, and where $\Gamma\subset X$ is a separating simple curve and $\Gamma'\subset X'$ is a non-separating one.
\end{bei}

\begin{bei}\label{example non-orientable}
    For every $k\ge 2$ there are $3k-2$ pairwise non-isometric closed connected hyperbolic surfaces $X_0,\dots,X_{3k-3}$ with $\chi(X_i)=2-2k$, and for all $i=0,\dots,3k-3$ a pants decomposition $\Gamma_i\subset X_i$ such that the hyperbolic marimbas
    $$(X_0,\Gamma_0),\dots,(X_{3k-3},\Gamma_{3k-3})$$ 
    are isomelodic.
\end{bei}

All those examples were about pairs of hyperbolic marimbas such that the underlying hyperbolic surfaces have the same volume. Note that we didn't assume that in Theorem \ref{sat measure theory}. Indeed, the said theorem also applies in the following case:

\begin{bei}\label{ex volume}
    There are two isomelodic marimbas $(X,\Gamma)$ and $(X',\Gamma')$ with $\chi(X)\neq\chi(X')$.
\end{bei}

After these examples, the reader might well be wondering if, other than some measure theoretic information, one gets anything at all from the melodies of random geodesics. One does:

\begin{sat}\label{sat finitely many cousins}
    For any orientable hyperbolic marimba $(X_0,\Gamma_0)$ there are finitely many hyperbolic marimbas $(X_1,\Gamma_1),\dots,(X_k,\Gamma_k)$ such that any other orientable hyperbolic marimba $(X',\Gamma')$ with $\chi(X')=\chi(X)$ which is isomelodic to $(X_0,\Gamma_0)$ is isometric to one of the $(X_i,\Gamma_i)$'s.   
\end{sat}

While Theorem \ref{sat finitely many cousins} applies to any hyperbolic marimba $(X,\Gamma)$, things look even rosier if the marimba in question is {\em generic}, by which we mean that $X$ avoids a negligible subset of moduli space. 

\begin{sat}\label{sat generically unique}
    If the underlying hyperbolic surface of an orientable hyperbolic marimba $(X_0,\Gamma_0)$ is generic, then any other orientable marimba $(X',\Gamma')$ with $\chi(X')=\chi(X_0)$ which is isomelodic to $(X_0,\Gamma_0)$  is isometric to $(X_0,\Gamma_0)$.
\end{sat}

\begin{bem}
    Both Theorem \ref{sat finitely many cousins} and Theorem \ref{sat generically unique} fail without the proviso that $\chi(X)=\chi(X')$.
\end{bem}

Theorem \ref{sat measure theory}, Theorem \ref{sat finitely many cousins}, Theorem \ref{sat generically unique}, and Examples \ref{example separating pairs}-\ref{ex volume} are the main results of this paper. 
\medskip

Let us now briefly describe the content of the subsequent sections, profiting at the same time to give some indications on the strategy of the proofs of the results we just mentioned.

In Section \ref{sec ergodic} we recall some basic facts about the geodesic flow of hyperbolic surfaces. In Section \ref{sec definition of melodies} we formally define the {\em melody} ${\bf m}_{(X,\Gamma)}(v)$ of the geodesic associated to a vector $v\in T^1X$ and prove that the frequencies \eqref{eq frequency exists} exist and are basically independent of $v$. We also observe that we can recover from such frequencies the lengths in $X$ of the individual components of $\Gamma$. In Section \ref{sec measure theory} we prove Theorem \ref{sat measure theory}. The basic idea is that (on its image) one can find a measurable inverse for the map which sends $v\in T^1X$ to its melody ${\bf m}_{(X,\Gamma)}(v)$. We derive the existence of such a measurable inverse from classic but somewhat sophisticated results from measure theory like the disintegration theorem and the Lusin-Novikov theorem. In Section \ref{sec examples} we explain Examples \ref{example separating pairs}-\ref{ex volume} and in Section \ref{sec teichmueller theory} we recall some facts about Teichm\"uller theory. The only new result there is Lemma \ref{lem analytic function}, a result asserting that the image of certain specific analytic map is an analytic set. In Section \ref{sec hear orthospectrum} we prove that the melody ${\bf m}_{(X,\Gamma)}(v)$ determines what we call the {\em $k$-step orthospectrum} of $(X,\Gamma)$. For example, the $1$-step orthospectrum is nothing other than the classical orthospectrum of $X\setminus\Gamma$. In Section \ref{sec finitely many cousins} we will deduce Theorem \ref{sat finitely many cousins} from this observation, together with a recent finiteness result of Le Quellec \cite{Nolwenn} for the set of compact surfaces with boundary which share the same orthospectrum. In the same paper she also proved that generically this set is reduced to a point. This result, together with a small extension that we discuss in the Appendix, plays a role in the proof of Theorem \ref{sat generically unique}. However, there is still enough to do for us to distribute the argument over the two final sections of this paper. The basic idea of the proof of Theorem \ref{sat generically unique} is to adapt--as also Le Quellec does--the argument that Wolpert \cite{Wolpert} used to prove that generically the unmarked length spectrum determines the metric of a hyperbolic surface.

\subsection*{App and website}
What lead us to this paper was the idea that one could associate a melody to geodesics in a hyperbolic surface $X$ decorated with a multicurve $\Gamma$, and the desire to understand what one would actually hear out of such melodies. Being able to actually listen was from the beginning an integral part of this project. We ended up with an app able to handle extremely long geodesics in relatively short time (length 10.000.000 in less than 1 minute) and from which we were able to get statistically relevant information. In fact, the output of the app helped us understand and solidify in our minds some of the phenomena we exploit to formally prove the theorems stated above. The original app, written in Matlab, is available at \cite{github_repo}, and we plan to put a permanent version on Zenodo before the publication of the paper.

While the app does what we wanted it to do, it is maybe less user-friendly than one would wish, and it is not completely clear how one would make it more so. This is why we decided to also make public the website \href{https://ludox73.github.io/HyperMarimba/story.html}{{\em HyperMarimba}} \cite{webpage} integrating the app and some explanations about the relation between what is being displayed and what we do in this paper. Should the reader think that they want to read past this point, it might be a good idea for them to spend 5 minutes on that website to get some intuition on what is going on.

To conclude we should point out that AI, and more specifically Anthropic's Claude~\cite{AnthropicClaude}, was used to translate the app from Matlab to JavaScript, so that the processes could be run on a browser.

\subsection*{Acknowledgements:} The first author wishes to thank the Max-Planck-Institut f\"ur Mathematik for its hospitality, together with all the people that made this possible. The second author wishes to thank his nephews Cloe and Marco, and his brother in law, Christian Farroni, solist of the {\em Barcelona Symphony Orchestra - the National Orchestra of Catalonia}. This paper would have been unthinkable without their musical expertise.

\section{Preliminaries on the geodesic flow}\label{sec ergodic}
In this section we recall a few facts about dynamics of the geodesic flow on a closed connected hyperbolic surface $X$. We refer to \cite{Santalo, Paternain}, and specially to \cite{Bekka} for background on this topic.

\subsection{The kinematic measure and Santal\'o's formula}
Consistently with our notation above, let
$$\rho^X:\BR\times T^1X\to T^1X,\ (t, v) \mapsto\rho^X_v(t)$$
be the geodesic flow of the hyperbolic surface $X$. As it is the case for every Riemannian manifold, the geodesic flow $\rho^X$ preserves a preferred measure $\kinematic_X$ in the Lebesgue class. This measure, called {\em kinematic measure}, can be obtained from the canonical measure on the cotangent bundle of the smooth manifold $TX$--when one does this, the invariance under the geodesic flow is Liouville's theorem from classical (Hamiltonian) mechanics. The kinematic measure $\kinematic_X$ can also be seen as the Riemannian measure of the restriction to $T^1X$ of the Sasaki metric \cite{Sasaki} on $TX$. Anyways, from either point of view one gets that the kinematic measure is the one satisfying 
$$\int_{T^1X} f(v)\ d\kinematic_X(v)=\int_X\left(\int_{T^1_xX}f(\theta)\ d\theta\right)\ dx$$
for any measurable function $f$ on $T^1X$. Here $dx$ stands for integrating with respect to the Riemannian measure on $X$ and $d\theta$ stands for integrating over the circle $T^1_xX$ parametrized to have length $2\pi$.

If $Y$ is a hyperbolic surface with boundary then the geodesic flow $\rho^Y$ is only defined on a subset of $\BR\times T^1X$, basically because when an orbit hits the boundary it makes no sense to continue it--think of Wile E. Coyote running over the edge of a cliff. For $v\in T^1Y$ denote by $\tau^Y(v)\in[0,\infty]$ the maximal time that the orbit $t\mapsto\rho^Y_v(t)$ is defined. Note that $\tau^Y(v)=0$ if $v$ is based on a boundary point and points outside. The map $\tau^Y$ takes the value $\infty$ in a closed set of vanishing measure, and it is analytic on the open set $(\tau^Y)^{-1}(0,\infty)$.

Still assuming that $Y$ has boundary, note that we can parametrize the set 
$$T^1_+\D Y=\{v\in T^1Y\vert_\Gamma\text{ with }\tau^Y(v)\ge 0\}$$
of inward pointing vectors based on the boundary by $\D Y\times[0,\pi]$. In particular we have preferred coordinates $(x,\theta)$ on $T^1_+\D Y$. The volume form $dx\wedge d\theta$ yields thus a canonical measure on $T^1_+\D Y$. Santal\'o's\footnote{Lu\'is Santal\'o was a Spanish mathematician who, like half a million others, left Spain because of the civil war. He was welcome in Argentina. Many others went to rebuild their lives in France or elsewhere in Latin America. This is the kind of thing that should not be forgotten.} formula \cite{Santalo, Berger} relates this measure to the kinematic measure on $T^1Y$. It asserts that we have
    $$\int_{T^1Y}f(v)\ d\kinematic_X(v)=\int_{T^1_+\D Y}\left(\int_0^{\tau(x,\theta)}f(\rho^Y_{(x,\theta)}(t))\ dt\right)\vert\sin(\theta)\vert\ d\theta dx$$
for every integrable function $f$ on $T^1Y$.

In our setting we will be working with closed surfaces $X$ decorated with a simple multicurve $\Gamma$. For $v\in T^1X$ denote by 
$$\tau^{(X,\Gamma)}(v)=\inf\{t>0\text{ such that }\rho_v^X(t)\text{ is based on }\Gamma\}$$
the first positive time (if any) that the geodesic flow orbit of $v$ hits $\Gamma$ and identify the restriction $T^1X\vert_\Gamma$ to $\Gamma$ of the unit tangent bundle of $X$ with $\Gamma\times\BS^1$, where the circle is once again parametrized by $[0,2\pi]$ in such a way that $0,\pi$ are tangent to $\Gamma$. Applying Santal\'o's formula to the metric completion of $X\setminus\Gamma$ we get
\begin{equation}\label{eq practical Santalo}
    \int_{T^1X}f(v)\ d\kinematic_X(v)=\int_{\Gamma}\left[\int_0^{2\pi}\left(\int_0^{\tau^{(X,\Gamma)}(x,\theta)}f(\rho^X_{(x,\theta)}(t))\ dt\right)\vert\sin(\theta)\vert\ d\theta\right] dx
\end{equation}
for every integrable function $f$ on $T^1X$.

\subsection{Ergodicity}
Recall now that, since $X$ is a closed hyperbolic surface, the geodesic flow $\rho^X$ does not only preserve the kinematic measure $\kinematic_X$, but that it is also ergodic: every measurable $\rho^X$-invariant subset of $T^1X$ has either zero or full $\kinematic_X$-measure. Now, when invoking the ergodicity of the geodesic flow, it turns out to be more convenient to replace the kinematic measure by the associated probability measure
$$\Liouville_X=\frac 1{4\pi^2\vert\chi(X)\vert}\kinematic_X.$$
For example, the ergodic theorem \cite{Bekka} asserts now that the $\Liouville_X$-integral of an integrable function $f$ on $T^1X$ can be computed by taking the integral over the orbit of $v$ for almost every $v\in T^1X$. Now, noting that the space of (automatically compactly supported) continuous functions on $T^1X$ is separable--that is, has a countable dense set--, we get that if we restrict ourselves to continuous functions then we can exchange the order of the quantifiers in the ergodic theorem, getting thus that {\em for almost every $v\in T^1X$ we have}
\begin{equation}\label{eq random vector}
    \int_{T^1X}f(v)\ d\Liouville_X(v)=\lim_{T\to\pm\infty}\frac 1T\int_0^Tf(\rho^X_v(t))\ dt\text{ for all }f\in  C^0(T^1X).
\end{equation}
For economy of language, let us introduce some terminology:

\begin{defi}
A vector $v\in T^1X$ is a {\em random vector} if \eqref{eq random vector} holds. We denote by 
\begin{equation}
\CR(X)=\{v\in T^1X\text{ satisfying \eqref{eq random vector}}\}
\end{equation}
the measurable set of random vectors in $X$.     
\end{defi}

We stress that $\Liouville_X(\CR(X))=1$.

Rather than integrating continuous functions, one sometimes wants to calculate the measure of open sets. From the Portmanteau theorem we get that if $U\subset T^1X$ is open with $\Liouville_X(\overline U\setminus U)=0$ where $\overline U$ is the closure of $U$ then we have
$$\Liouville_X(U)=\lim_{T\to\infty}\frac 1T\int_0^T\chi_U(\rho^X_v(t))\ dt$$
for every random vector $v\in\CR(X)$. Here, $\chi_U$ is the characteristic function of $U$.

We will apply this next to a particular kind of open sets, but first we need a bit of notation. Recalling that we have identified $T^1X\vert_\Gamma\simeq\Gamma\times\BS^1\simeq\Gamma\times[0,2\pi]$, let
$$T^1_X\vert_\Gamma^*=\Gamma\times((0,\pi)\cup(\pi,2\pi))$$
be the set of vectors based on $\Gamma$ which are not tangent to $\Gamma$. Let also
\begin{equation}
    \tau^{(X,\Gamma)}_{\min}=\min\{\tau^{(X,\Gamma)}(v)\text{ with }v\in T^1X\vert_\Gamma^*\}
\end{equation}
be the length of the shortest essential geodesic arc with end points in $\Gamma$. We will say that $\epsilon>0$ is {\em small} if $\epsilon<\tau^{(X,\Gamma)}_{\min}$. The point of having this terminology in place is that we get that for every $V\subset T^1_X\vert_\Gamma^*$ open and every $\epsilon>0$ small, the map
\begin{equation}\label{eq push via geod flow}
    (0,\epsilon)\times V\to T^1X,\ (t,v)\mapsto\rho_v(t)
\end{equation}
is a homeomorphism onto an open set $U(V,\epsilon)$. 

\begin{lem}\label{lem distribution of vectors over time}
    For every $V\subset T^1X\vert_\Gamma^*$ open with $\int_{\overline V}1d\theta dx=\int_V1d\theta dx$ we have 
    $$\frac 1{4\pi^2\vert\chi(X)\vert}\int_V\vert\sin(\theta)\vert\ d\theta dx=\lim_{T\to\infty}\frac 1T\cdot\vert\{t\in[0,T]\text{ with }\rho_v(t)\in V\}\vert$$
    for every random vector $v\in\CR(X)$.
\end{lem}
\begin{proof}
    The assumption $\int_{\overline V}1d\theta dx=\int_V1d\theta dx$ implies that, for small $\epsilon>0$, the boundary of the image $U(V,\epsilon)$ of \eqref{eq push via geod flow} has vanishing measure. We get thus from the ergodic theorem that 
    \begin{align*}
        \Liouville_X(U(V,\epsilon))
        &=\lim_{T\to\infty}\frac 1T\int_0^T\chi_{U(V,\epsilon)}(\rho^X_v(t))\ dt\\
        &=\lim_{T\to\infty}\frac 1T\cdot\epsilon\cdot\vert\{t\in[0,T]\text{ with }\rho_v(t)\in V\}\vert.
    \end{align*}
    On the other hand, we get from Santal\'o's formula that 
    \begin{align*}
        \Liouville_X(U(V,\epsilon))
        &=\frac 1{4\pi^2\vert\chi(X)\vert}\kinematic_X(U(V,\epsilon))\\
        &=\frac 1{4\pi^2\vert\chi(X)\vert}\int_V\epsilon\cdot\vert\sin(\theta)\vert\ d\theta dx.
    \end{align*}
    The claim follows from this two equalities.
\end{proof}

When we apply this to the subset $T_X^1\vert_\gamma^*$ of $T^1_X\vert_\Gamma^*$ consisting of vectors based in some component $\gamma$ of $\Gamma$ we get that for every $v\in\CR(X)$ we have 
\begin{equation}\label{eq read length out of hits}
    \begin{split}
        \lim_{T\to\infty}\frac 1T\vert\{t\in[0,T]\text{ with }\rho^X_v(t)\text{ based on }\gamma\}\vert
    &=\frac 1{4\pi^2\vert\chi(X)\vert}\int_{T_X^1\vert_\gamma^*}\vert\sin(\theta)\vert\ d\theta dx\\
    &=\frac 1{\pi^2\vert\chi(X)\vert}\ell_X(\gamma)
    \end{split}
\end{equation}
where $\ell_X(\gamma)$ is the length in $X$ of the geodesic $\gamma$. In particular, when we denote by 
\begin{equation}\label{eq note counting}
  N_\Gamma(v,T)=\vert\{t\in[0,T]\text{ with }\rho^X_v(t)\text{ based on }\Gamma\}\vert  
\end{equation}
we get for every $v\in\CR(X)$ that
\begin{equation}\label{eq note counting proportional to time}
  N_\Gamma(v,T)\sim \frac 1{\pi^2\vert\chi(X)\vert}\ell_X(\Gamma)\cdot T  
\end{equation}
where $\sim$ means that the ratio between both sides tends to $1$ as $T\to\infty$.

\section{Melodies and frequencies of motifs}\label{sec definition of melodies}
In this section we formalize what we understand by the {\em melody} of a geodesic in a hyperbolic marimba, prove the frequencies \eqref{eq frequency exists} exist, and observe that we can read out of these frequencies the lengths of the individual components of $\Gamma$. All this just needs a little bit of ergodic theory. Let us however start with some terminology: under a {\em topological marimba} we understand a pair $(X,\Gamma)$ where $X$ is a closed connected surface and $\Gamma$ is a simple multicurve, that is a finite union of disjoint, simple, homotopically essential curves which are pairwise not homotopic to each other. If the surface $X$ is endowed with some extra structure, like for example being oriented or hyperbolic, then we say that $(X,\Gamma)$ is an oriented or hyperbolic marimba. If the marimba $(X,\Gamma)$ is hyperbolic, we will assume that $\Gamma$ is geodesic.

\begin{bem}
Although it is basically not used, we remind the reader that we are just considering closed marimbas $(X,\Gamma)$. We also stress that the underlying surface $X$ is assumed to be connected.
\end{bem}

Suppose that $(X,\Gamma)$ is a hyperbolic marimba and let $v\in\CR(X)$ be a random unit tangent vector--recall that this means that $v$ satisfies \eqref{eq random vector}. Noting that the geodesic $t\mapsto\rho^X_v(t)$ meets $\Gamma$ infinitely often, we get a well-defined bi-infinite sequence 
$$\dots<t_{-2}^\Gamma(v)<t_{-1}^\Gamma(v)<t_{0}^\Gamma(v)<t_1^\Gamma(v)<t_2^\Gamma(v)<\dots$$
with $0\in(t_{-1}^\Gamma(v),t_{0}^\Gamma(v)]$ and such that $\{t\in\BR,\ \rho_v(t)\in\Gamma\}=\{t_i^\Gamma(v),i\in\BZ\}$. Moreover, abusing terminology and denoting also by the same symbol the collection of curves $\Gamma$ and the set of its individual components, we also get a bi-infinite sequence $(\kappa_i^\Gamma(v))_{i\in\BZ}\in\Gamma^\BZ$ with 
$$\rho_v(t_i^\Gamma(v))\text{ based at }\kappa_i^\Gamma(v)\in\Gamma.$$
In this way we get a continuous map
\begin{equation}\label{eq partiture function}
{\bf m}_{(X,\Gamma)}:\CR(X)\to(\Gamma\times\BR)^\BZ,\ v\mapsto(\kappa_i^\Gamma(v),t_{i}^\Gamma(v))_i.
\end{equation}
In musical terms, we think of ${\bf m}_{(X,\Gamma)}(v)$ as a score that tells us that at time $t_i^\Gamma(v)$ we have to play the key $\kappa_i^\Gamma(v)$. 

\begin{defi}
Let $(X,\Gamma)$ be a hyperbolic marimba. The sequence ${\bf m}_{(X,\Gamma)}(v)$ is the {\em melody of the random vector $v\in\CR(X)$}. The projection ${\bf m}^+_{(X,\Gamma)}(v)$ of ${\bf m}_{(X,\Gamma)}(v)$ to $(\Gamma\times\BR)^\BN$ is the {\em positive melody of $v$}.
\end{defi}

\begin{bem}
    What one can listen to in \cite{webpage} is the positive melody.
\end{bem}

The following result basically asserts that, locally, $v$ is determined by its melody ${\bf m}(v)$:

\begin{lem}\label{lem discreteness of music}
    The map ${\bf m}_{(X,\Gamma)}:\CR(X)\to(\Gamma\times\BR)^\BZ$ is finite-to-one.
\end{lem}
\begin{proof}
    It suffices to prove that for any three disjoint geodesics $\gamma_{-},\gamma_0,\gamma_{+}\subset\BH^2$, and for any two $s,t>0$ the set
    $$A=\{v\in T^1\BH^2\vert_{\gamma_0}\text{ with }\rho_v^{\BH^2}(-s)\in\gamma_{-}\text{ and }\rho_v(t)\in\gamma_{+}\}$$
    is finite. Here $T^1\BH^2\vert_{\gamma_0}\simeq\gamma_0\times\BS^1$ is the restriction of the unit tangent bundle of $\BH^2$ to $\gamma_0$. The key observation is that both sets $\{v\in T^1\BH^2\vert_{\gamma_0}\text{ with }\rho_v^{\BH^2}(-s)\in\gamma_{-}\}$ and $\{v\in T^1\BH^2\vert_{\gamma_0}\text{ with }\rho_v^{\BH^2}(t)\in\gamma_{+}\}$ are irreducible proper analytic subsets of the cylinder $T^1\BH^2\vert_{\gamma_0}$, and that, unless they are both empty or reduced to a point, they are distinct. It follows that their intersection is finite, as claimed. 
\end{proof}

On the space $(\Gamma\times\BR)^\BZ$ we have the flow, 
\begin{align*}
\sigma:\BR\times (\Gamma\times\BR)^\BZ&\to (\Gamma\times\BR)^\BZ\\
(s,(\kappa_i,t_i)_i)&\mapsto\sigma_s((\kappa_i,t_i)_i)=(\kappa_{i+\mu(s,(t_i)_i))},t_{i+\mu(s,(t_i)_i))}-s),
\end{align*}
where $\mu(s,(t_i)_i)=j\in\BZ$ if $s\in(t_{j-1},t_j]$. This flow, of which we think as a shift, is related to the geodesic flow. In fact, the map ${\bf m}_{(X,\Gamma)}$ is a semi-conjugacy onto its image:
$${\bf m}_{(X,\Gamma)}(\rho_s(v))=\sigma_s({\bf m}_{(X,\Gamma)}(v)).$$
In particular, it follows from the invariance of the probability measure $\Liouville_X$ under the geodesic flow that the flow $\sigma$ preserves its push-forward
$$\FM_{(X,\Gamma)}=({\bf m}_{(X,\Gamma)})_*(\Liouville_X).$$ 
The measure $\FM_{(X,\Gamma)}$ is a probability measure because $\Liouville_X(\CR(X))=1$. Moreover, the ergodicity of $\Liouville_X$ implies that the measure $\FM_{(X,\Gamma)}$ is ergodic as well. Let us record these facts.

\begin{lem}
    The measure $\FM_{(X,\Gamma)}$ on $(\Gamma\times\BR)^\BZ$ is an invariant ergodic probability measure for the flow $\sigma$.\qed
\end{lem}

\begin{bem}
    It would be interesting to calculate the entropy of the measure $\FM_{(X,\Gamma)}$, to evaluate how much information is lost when we pass from vectors $v\in\CR(X)$ to their melodies ${\bf m}_{(X,\Gamma)}(v)$.
\end{bem}

For every motif $[(\eta_i)_{i=0,\dots,k},(t_i)_{i=1,\dots,k},\epsilon]$ and every $\delta\in(0,\tau^{(X,\Gamma)}_{\min})$ positive and small, we have the open set 
$$U_{[(\eta_i)_i,(t_i)_i,\epsilon]}(\delta)=
\left\{v\in\CR(X)\middle\vert\begin{array}{l}\text{there are }t\in(0,\delta)\text{ and }0=s_0<s_1<s_2<\dots<s_k\\
\text{with }s_i\in(t_i,t_i+\epsilon)\text{ for }i=1,\dots,k,\\
\rho^X_v(t+s_i)\in\eta_i\text{ for }i=0,\dots,k\text{, and}\\
\{0,\dots,s_k\}=\{s\in[0,s_k]\text{ with }\rho_v^X(s+t)\in\Gamma\}.
\end{array}\right\}$$
consisting of those $v\in T^1X$ such that the geodesic $\rho^X_v$ plays the motif $[(\eta_i)_i,(t_i)_i,\epsilon]$ at some time $t\in(0,\delta)$. The image of $U_{[(\eta_i)_i,(t_i)_i,\epsilon]}(\delta)$ under ${\bf m}_{(X,\Gamma)}$ is the intersection of ${\bf m}_{(X,\Gamma)}(\CR(X))$ with the "cylinder" set 
$$V_{[(\eta_i)_i,(t_i)_i,\epsilon]}(\delta)=\left\{(\kappa_i,\hat t_i)_i\in(\Gamma\times\BR)^{\BZ}\middle\vert
\begin{array}{l}\kappa_i=\eta_i\text{ for }i=0,\dots,k\text{ and such that}\\
\text{there are }t\in(0,\delta)\text{ and }0=s_0<s_1<s_2<\dots<s_k\\
\text{with }s_i\in(t_i,t_i+\epsilon)\text{ for }i=1,\dots,k,\\
\text{and with }\hat t_i=t+s_i\text{ for }i=0,\dots,k.
\end{array}\right\}$$
Since $V_{[(\eta_i)_i,(t_i)_i,\epsilon]}(\delta)$ is the pre-image of an open set in $(\Gamma\times\BR)^{\{0,\dots,k\}}$ under the projection
$$(\Gamma\times\BR)^{\BZ}\to(\Gamma\times\BR)^{\{0,\dots,k\}},$$
it is open. Indeed, when we consider all motifs $[(\eta_i)_{i=0,\dots,k},(t_i)_{i=1,\dots,k},\epsilon]$, all $\delta>0$, and all $s\in\BR$, we get that the sets $\sigma_s(V_{[(\eta_i)_i,(t_i)_i,\epsilon]}(\delta))$ form a basis of the topology of $(\Gamma\times\BR)^{\BZ}$. It follows that the measure $\FM_{(X,\Gamma)}$ is determined by the measure 
$$\FM_{(X,\Gamma)}(\sigma_s(V_{[(\eta_i)_i,(t_i)_i,\epsilon]}(\delta)))=\FM_{(X,\Gamma)}(V_{[(\eta_i)_i,(t_i)_i,\epsilon]}(\delta))=\Liouville_X(U_{[(\eta_i)_i,(t_i)_i,\epsilon]}(\delta)).$$
of such sets. On the other hand we get from Lemma \ref{lem distribution of vectors over time} that for every $v\in\CR(X)$ we have
\begin{align*}
\Liouville_X(U_{[(\eta_i)_i,(t_i)_i,\epsilon]}(\delta))
&=\lim_{T\to\infty}\frac 1T\Leb\left(\left\{ t\in[0,T]\middle\vert
\begin{array}{l}
\text{there is }s\in(0,\delta)\text{ such that the motif}\\ 
\left[(\eta_i),(t_i),\epsilon\right]\text{ is played at time }t+s\\
\text{in the melody of }\rho_v^X\end{array}
\right\}\right)\\
&=\delta\lim_{T\to\infty}\frac 1T\#\left\{t\in[0,T]\middle\vert\begin{array}{l}
\text{ the motif }[(\eta_i),(t_i),\epsilon]\text{ is played}\\ \text{ at time }t\text{ in the melody of }\rho_v^X\end{array}\right\}.
\end{align*}
Recalling the definition \eqref{eq frequency exists} of the frequency $\FF^{(X,\Gamma)}_{[(\eta_i),(t_i),\epsilon]}$ we have thus that
$$\Liouville_X(U_{[(\eta_i)_i,(t_i)_i,\epsilon]}(\delta))=\delta\cdot \FF^{(X,\Gamma)}_{[(\eta_i),(t_i),\epsilon]}$$
for all $\delta>0$ sufficiently small. In particular, the limit \eqref{eq frequency exists} exists. Combining everything we said so far, we have proved:

\begin{prop}\label{prop fequencies tell it all}
    The frequency $\FF^{(X,\Gamma)}_{[(\eta_i),(t_i),\epsilon]}$ associated to a motif $[(\eta_i),(t_i),\epsilon]$, that is the limit \eqref{eq frequency exists}, exists for every random vector $v\in\CR(X)$ and its value is independent of the particular $v$ chosen to calculate it. Moreover, the map which sends motifs to their frequencies uniquely determines the measure $\FM_{(X,\Gamma)}$.\qed
\end{prop}

To conclude this section let us observe that, if we know $\chi(X)$, we can read the lengths of the individual components $\gamma$ of $\Gamma$ out of the measure $\FM_{(X,\Gamma)}$. Note first that if a motif consists of a single note, then the parameter $\epsilon$ plays absolutely no role. We call such motifs {\em degenerate} and denote them by $[\gamma,0]=[\gamma,0,\epsilon]$. Anyways, invoking \eqref{eq read length out of hits} we get
\begin{align*}
    \FF^{(X,\Gamma)}_{[\gamma,0]}
    &=\lim_{T\to\infty}\frac 1T\#\left\{t\in[0,T]\middle\vert\begin{array}{l}
\text{ the motif }[\gamma,0]\text{ is played}\\ \text{ at time }t\text{ in the melody of }\rho_v^X\end{array}\right\}\\    
    &=\lim_{T\to\infty}\frac 1T\vert\{t\in[0,T]\text{ with }\rho^X_v(t)\text{ based on }\gamma\}\vert\\
    &=\frac 1{\pi^2\vert\chi(X)\vert}\ell_X(\gamma)
\end{align*}
From here we get:

\begin{lem}\label{lem hear length}
    We have $\ell_X(\gamma)=\pi^2\vert\chi(X)\vert\cdot\FF^{(X,\Gamma)}_{[\gamma,0]}$ for any $\gamma\in\Gamma$.\qed
\end{lem}

\section{Proof of Theorem \ref{sat measure theory}}\label{sec measure theory}
Noting that both $T^1X$ and $(\Gamma\times\BR)^\BZ$ are Polish spaces (that is, that they have countable dense sets and are homeomorphic to complete metric spaces) we can apply the {\em disintegration theorem} \cite[Theorem 14.D.10]{Baccelli} to the map
$${\bf m}_{(X,\Gamma)}:\CR(X)\to(\Gamma\times\BR)^\BZ$$
and the measure $\Liouville_X$. It follows that there is an essentially unique Borel map
\begin{equation}\label{eq map disintegration}
\mu:(\Gamma\times\BR)^\BZ\to\CP(\CR(X)),\ \mu:(\kappa_i,t_i)_i\mapsto\mu_{(\kappa_i,t_i)_i}
\end{equation}
to the space $\CP(\CR(X))$ of Borel probability measures on $\CR(X)$ such that for all $E\subset\CR(X)$ measurable we have 
$$\mu_{(\kappa_i,t_i)_i}(E)=\mu_{(\kappa_i,t_i)_i}(E\cap{\bf m}_{(X,\Gamma)}^{-1}((\kappa_i,t_i)_i))$$ 
and such that for every measurable function $f:\CR(X)\to[0,\infty]$ we have
\begin{equation}\label{eq disintegration}
    \int_{\CR(X)}f(v)\ d\Liouville_X(v)=\int_{(\Gamma\times\BR)^{\BZ}}\left(\int_{{\bf m}_{(X,\Gamma)}^{-1}((\kappa_i,t_i)_i)}f(v)\ d\mu_{(\kappa_i,t_i)_i}\right)d\FM_{(X,\Gamma)}((\kappa_i,t_i)_i).
\end{equation}
Here, {\em essentially unique} means that any two such maps $(\kappa_i,t_i)_i\mapsto\mu_{(\kappa_i,t_i)_i}$ agree on set of full measure with respect to $\FM_{(X,\Gamma)}=({\bf m}_{(X,\Gamma)})_*(\Liouville_X)$. Uniqueness has the following consequence:

\begin{lem}\label{lem invariance of disintegration}
    For all $t\in\BR$ we have
    \begin{equation}\label{eq invariance of disintegration}
        \mu_{\sigma_t((\kappa_i,t_i)_i)}=(\rho_t)_*\mu_{(\kappa_i,t_i)_i}
    \end{equation}
    for $\FM_{(X,\Gamma)}$-almost every $(\kappa_i,t_i)_i\in(\Gamma\times\BR)^\BZ$.
\end{lem}
\begin{proof}
    Invariance of the Liouville measure under the geodesic flow $\rho=\rho^X$, invariance of $\FM_{(X,\Gamma)}$ under $\sigma$, and the fact that ${\bf m}_{(X,\Gamma)}$ is a semi-conjugacy between $\rho$ and $\sigma$, imply that 
    \begin{align*}
        \int_{\CR(X)}f(v)\ d\Liouville_X(v)
        &=\int_{\CR(X)}(f\circ\rho_t)(v)\ d\Liouville_X(v)\\
        &\stackrel{\eqref{eq disintegration}}=\int_{(\Gamma\times\BR)^{\BZ}}\left(\int_{{\bf m}_{(X,\Gamma)}^{-1}((\kappa_i,t_i)_i)}(f\circ\rho_t)(v)\ d\mu_{(\kappa_i,t_i)_i}\right)d\FM_{(X,\Gamma)}((\kappa_i,t_i)_i)\\
        &=\int_{(\Gamma\times\BR)^{\BZ}}\left(\int_{\rho_t({\bf m}_{(X,\Gamma)}^{-1}((\kappa_i,t_i)_i))}f(v)\ d(\rho_t^{-1})_*(\mu_{(\kappa_i,t_i)_i})\right)d\FM_{(X,\Gamma)}((\kappa_i,t_i)_i)\\
        &=\int_{(\Gamma\times\BR)^{\BZ}}\left(\int_{{\bf m}_{(X,\Gamma)}^{-1}(\sigma_t^{-1}((\kappa_i,t_i)_i))}f(v)\ d(\rho_t^{-1})_*(\mu_{\sigma_t(\sigma_t^{-1}((\kappa_i,t_i)_i))})\right)d\FM_{(X,\Gamma)}((\kappa_i,t_i)_i)\\
        &=\int_{(\Gamma\times\BR)^{\BZ}}\left(\int_{{\bf m}_{(X,\Gamma)}^{-1}((\kappa_i,t_i)_i)}f(v)\ d(\rho_t^{-1})_*(\mu_{\sigma_t((\kappa_i,t_i)_i)})\right)d(\sigma_t)_*\FM_{(X,\Gamma)}((\kappa_i,t_i)_i)\\
        &=\int_{(\Gamma\times\BR)^{\BZ}}\left(\int_{{\bf m}_{(X,\Gamma)}^{-1}((\kappa_i,t_i)_i)}f(v)\ d(\rho_t^{-1})_*(\mu_{\sigma_t((\kappa_i,t_i)_i)})\right)d\FM_{(X,\Gamma)}((\kappa_i,t_i)_i)\\
    \end{align*}
    It follows that the map
    $$(\Gamma\times\BR)^\BZ\to\CP(\CR(X)),\ (\kappa_i,t_i)_i\mapsto(\rho_t^{-1})_*(\mu_{\sigma_t((\kappa_i,t_i)_i)})$$
    also satisfies \eqref{eq disintegration}. The claim follows thus from the uniqueness of \eqref{eq map disintegration}.
\end{proof}

We claim next that measures $\mu_{(\kappa_i,t_i)_i}$ are atomic, with a constant number of atoms, all of the same weight:

\begin{prop}\label{lem kappa to one}
    For every hyperbolic marimba $(X,\Gamma)$ there is some $\kappa=\kappa_{(X,\Gamma)}$ such that:
    \begin{itemize}
        \item For $\Liouville_X$-almost every $v\in\CR(X)$ there are exactly $\kappa$ distinct vectors $v_1,\dots,v_\kappa$ with ${\bf m}_{(X,\Gamma)}(v_i)={\bf m}_{(X,\Gamma)}(v)$. 
        \item Moreover, the probability measures $\mu_{(\kappa_i,t_i)_i}$ in \eqref{eq disintegration} consist $\FM_{(X,\Gamma)}$-almost surely of atoms of equal weight $\frac 1\kappa$.
    \end{itemize}
\end{prop}

\begin{proof}
    The probability measures $\mu_{(\kappa_i,t_i)_i}$ are supported by ${\bf m}_{(X,\Gamma)}^{-1}((\kappa_i,t_i)_i)$, a discrete set by Lemma \ref{lem discreteness of music}. It follows that $\mu_{(\kappa_i,t_i)_i}$ is purely atomic. Now, we get from Lemma \ref{lem invariance of disintegration} and the fact that ${\bf m}_{(X,\Gamma)}$ is a semiconjugacy between $\rho$ and $\sigma$ that 
    \begin{align*}
        \mu_{{\bf m}_{(X,\Gamma)(\rho_t(v))}}(\{\rho_t(v)\})
        &=\mu_{\sigma_t({\bf m}_{(X,\Gamma)}(v))}(\{\rho_t(v)\})=\big((\rho_t)_*\mu_{{\bf m}_{(X,\Gamma)}(v)}\big)(\{\rho_t(v)\})\\
        &=\mu_{{\bf m}_{(X,\Gamma)}(v)}(\{v\}).    
    \end{align*}
    In other words, the function
    $$\CR(X)\to\BR,\ v\mapsto\mu_{{\bf m}_{(X,\Gamma)}(v)}(v)$$
    is invariant under the geodesic flow, and hence essentially constant. It follows that there is some $\epsilon\ge 0$ with $\mu_{{\bf m}_{(X,\Gamma)}(v)}(v)=\epsilon$ for $\Liouville$-almost every $v\in\CR(X)$. This implies that 
    $$1=\mu_{(x_i,t_i)_i}({\bf m}_{(X,\Gamma)}^{-1}((\kappa_i,t_i)_i)=\sum_{v\in {\bf m}_{(X,\Gamma)}^{-1}((\kappa_i,t_i)_i}\mu_{{\bf m}_{(X,\Gamma)}(v)}(v)=\epsilon\cdot\left\vert {\bf m}_{(X,\Gamma)}^{-1}((\kappa_i,t_i)_i\right\vert$$
    for $\FM_{(X,\Gamma)}$-almost every $(\kappa_i,t_i)_i$. The claim follows when we set $\kappa=\frac 1\epsilon$.
\end{proof}

We are now ready to prove Theorem \ref{sat measure theory}, which we restate here for the convenience of the reader:

\begin{named}{Theorem \ref{sat measure theory}}
    Two hyperbolic marimbas $(X,\Gamma)$ and $(X',\Gamma')$ are isomelodic if and only if there is a full measure set $U\subset T^1X$ and an absolutely continuous map
    $$\Phi:U\to T^1X'$$
    such that for all $v\in U$ the $\Gamma$-melody of the geodesic in $X$ corresponding to $v$ is exactly the same one that the $\Gamma'$-melody of the geodesic in $X'$ corresponding to $\Phi(v)$: exactly the same notes being played at exactly the same times.
\end{named}

In the course of the proof we will be using the {\em Lusin-Novikov Uniformization Theorem} \cite[Thm. 18.10]{Kechris}, which we recall now. This fundamental theorem says that {\em if $\CX$ and $\CY$ are Polish spaces and $\CE\subseteq\CX \times\CY$ is a Borel set such that every vertical section $\CE_x =\{y \in\CY : (x, y) \in\CE\}$ is countable, then there exists a Borel set $\CE^* \subseteq\CE$ such that $\CE^*$ intersects each nonempty section $\CE_x$ in exactly one point. Moreover, the projections $\CE^*\to\CX$ and $\CE^*\to\CY$ are Borel.} Recall that every Borel subset of a Polish space is a standard Borel space.

\begin{proof}
Suppose first that the map $\Phi$ exists and note that $\Phi(U\cap\CR(X))$ has positive measure on $T^1X'$. It follows that $\Phi(U\cap\CR(X))\cap\CR(X')\neq\emptyset$. This means that there is $v\in\CR(X)$ such that $\Phi(v)\in\CR(X')$. This means that, by Proposition \ref{prop fequencies tell it all}, we can use the melodies ${\bf m}_{(X,\Gamma)}(v)$ and ${\bf m}_{(X',\Gamma')}(\Phi(v))$ to calculate the frequencies of motifs in $(X,\Gamma)$ and $(X',\Gamma')$. Since these melodies are identical, the frequencies of every motif in $(X,\Gamma)$ and $(X',\Gamma')$ agree, meaning that both marimbas are isomelodic.

To establish the other implication we are going to apply the Lusin-Novikov Uniformization Theorem to 
$$\CY=\CR(X'),\ \CX={\bf m}_{(X',\Gamma')}(\CR(X'))\text{ and }\CE=\{({\bf m}_{(X',\Gamma')}(v),v)\text{ for }v\in\CR(X')\}\subset\CX\times\CY.$$
From Lemma \ref{lem discreteness of music} we get that for all $x\in\CX$ the vertical section $\CE_x =\{y \in\CY : (x, y) \in\CE\}$ is finite and hence countable. It follows thus from Lusin-Novikov that there is a Borel set $\CE^*\subset\CE$ such that for all $(\kappa_i,t_i)_i\in{\bf m}_{(X',\Gamma')}(\CR(X'))$ there is a unique $v_{(\kappa_i,t_i)_i}\in\CR(X')$ with $((\kappa_i,t_i)_i,v_{(\kappa_i,t_i)_i})\in\CE^*$. Since the domain and target are Polish spaces and since its graph is measurable, we get that the map
$$L_{(X',\Gamma')}:{\bf m}_{(X',\Gamma')}(\CR(X'))\to\CR(X'),\ (\kappa_i,t_i)_i\to v_{(\kappa_i,t_i)_i}$$
is measurable. Since it is a bijection onto its image, it is a Borel isomorphism. 

Moreover we get from Proposition \ref{lem kappa to one} that there is some $\kappa>0$ such that for all measurable functions $f:T^1X'\to[0,\infty]$ we have
\begin{align*}
    \int_{\CR(X)}f(v)\ d\Liouville_X(v)
    &=\int_{(\Gamma\times\BR)^{\BZ}}\left(\int_{{\bf m}_{(X,\Gamma)}^{-1}((\kappa_i,t_i)_i)}f(v)\ d\mu_{(\kappa_i,t_i)_i}\right)d\FM_{(X,\Gamma)}((\kappa_i,t_i)_i)\\
    &\ge \int_{(\Gamma\times\BR)^{\BZ}}f(L_{(X',\Gamma')}((\kappa_i,t_i)_i))\mu_{(\kappa_i,t_i)_i}(L_{(X',\Gamma')}((\kappa_i,t_i)_i))\ d\FM_{(X,\Gamma)}((\kappa_i,t_i)_i)\\
    &=\frac 1\kappa\int_{(\Gamma\times\BR)^{\BZ}}f(L_{(X',\Gamma')}((\kappa_i,t_i)_i))\ d\FM_{(X,\Gamma)}((\kappa_i,t_i)_i)\\
    &=\frac 1\kappa\int_{(\Gamma\times\BR)^{\BZ}}(f\circ L_{(X',\Gamma')})\ d\FM_{(X,\Gamma)}.
\end{align*}
It follows in particular that $L_{(X',\Gamma')}$ is absolutely continuous. 

Now, we are assuming that $(X,\Gamma)$ and $(X',\Gamma')$ are isomelodic, meaning that the frequencies of all motifs agree:
$$\FF^{(X,\Gamma)}_{[(\eta_i),(t_i),\epsilon]}=\FF^{(X',\Gamma')}_{[(\eta_i),(t_i),\epsilon]}\text{ for all motifs }[(\eta_i),(t_i),\epsilon].$$
We get thus from Proposition \ref{prop fequencies tell it all} that the measures $\FM_{(X,\Gamma)}$ and $\FM_{(X',\Gamma')}$ agree. It follows in particular that their supports are identical. In particular, there is a full measure subset $U\subset\CR(X)\subset T^1X$ with 
$${\bf m}_{(X,\Gamma)}(U)\subset{\bf m}_{(X',\Gamma')}(\CR(X')).$$
The map
$$\Phi:U\to T^1X',\ v\mapsto L_{(X',\Gamma')}({\bf m}_{(X,\Gamma)}(v))$$
is the map we were after.
\end{proof}

Example \ref{example non-orientable} shows that in general the map $\Phi$ provided by Theorem \ref{sat measure theory} cannot be chosen to be measure preserving. However, this might be possibly true if one further assumes that $\chi(X)=\chi(X')$.

\begin{quest}\label{quest measure preserving}
    Suppose that $(X,\Gamma)$ and $(X',\Gamma)$ are isomelodic and satisfy $\chi(X)=\chi(X')$. Are there a full measure set $U\subset T^1X$ and a measure preserving map $\Phi: U\to T^1X'$ such that for all $v\in U$ the $\Gamma$-melody of the geodesic in $X$ corresponding to $v$ is exactly the same one that the $\Gamma'$-melody of the geodesic in $X'$ corresponding to $\Phi(v)$?
\end{quest}

In some sense Question \ref{quest measure preserving} amounts to asking whether generic fibers of the maps 
$${\bf m}_{(X,\Gamma)}:\CR(X)\to(\Gamma\times\BR)^\BZ\text{ and }{\bf m}_{(X',\Gamma')}:\CR(X')\to(\Gamma'\times\BR)^\BZ$$ 
have the same cardinality $\kappa=\kappa'$ whenever $(X,\Gamma)$ and $(X',\Gamma')$ are isomelodic and satisfy $\chi(X)=\chi(X')$. A positive answer to Question \ref{question2} in the next section would imply a positive answer to Question \ref{quest measure preserving}.

\section{Examples}\label{sec examples}
In this section we discuss Examples \ref{example separating pairs}-\ref{ex volume} from the introduction. We present them in a different order, prioritizing some similarities in the constructions.

\subsection{Genus 2, non-separating pair}
\begin{named}{Example \ref{example separating pairs}}
    There are two isomelodic hyperbolic marimbas $(X,\Gamma)$ and $(X,\Gamma')$ where $X,X'$ are non-isometric surfaces of genus $2$, and both $\Gamma\subset X$ and $\Gamma'\subset X'$ are non-separating pairs.
\end{named}

\begin{figure}[h]
\leavevmode \SetLabels
\endSetLabels
\begin{center}
\AffixLabels{\centerline{\includegraphics[width=0.4\textwidth]{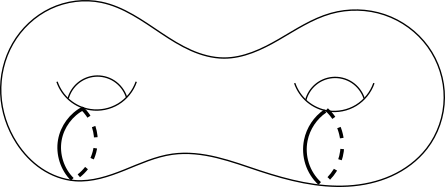}}\hspace{1cm}}
\end{center}
\caption{A non-separating pair.}
\end{figure}

Later on we will choose the surfaces in a specific $4$-dimensional family $\CZ$, but let us start with an arbitrary closed hyperbolic surface $X$ of genus $2$ and let $\Gamma$ be a non-separating multicurve consisting of two components $\alpha$ and $\beta$. Cutting $X$ along $\Gamma$ we obtain a four-holed sphere $W$. To recover $X$ from $W$, two components $\D_\alpha W$ and $\D_\alpha' W$ of $\D W$ get isometrically identified to produce $\alpha$, and another two $\D_\beta W$ and $\D_\beta' W$ to produce $\beta$. This evidently means that both $\D_\alpha W$ and $\D_\alpha' W$ (resp. $\D_\beta W$ and $\D_\beta' W$) have the same length $\ell_\alpha$ (resp. $\ell_\beta$). Conversely, when two positive numbers $\ell_\alpha$ and $\ell_\beta$ are fixed, the following holds:
\begin{itemize}
    \item[(*)] There is a unique hyperbolic metric on $W$ with totally geodesic boundary of length $\ell(\D_\alpha W)=\ell(\D_\alpha' W)=\ell_\alpha$ and $\ell(\D_\beta W)=\ell(\D_\beta' W)=\ell_\beta$, and with the property that there is a fixed-point free isometry $\tau: W \to W$, of order $2$, which fixes both $\D_\alpha W$ and $\D_\alpha' W$ as sets, and which interchanges $\D_\beta W$ and $\D_\beta' W$.
\end{itemize}
Indeed, this is the metric such that $W/\tau$ is the pair of pants with boundary lengths $\frac 12\ell_\alpha,\frac 12\ell_\alpha,\ell_\beta$. Note that $\tau$ preserves $\D_\alpha W$ (resp. $\D_\alpha' W$) as a set, but that it acts on it as a rotation by half of the length.

Let $\CW$ be the set of hyperbolic metrics on the 4-holed sphere $W$ satisfying (*). Let also $\CZ$ be the set of hyperbolic metrics $X$ on $\Sigma$ such that when we cut $X$ along the geodesic representatives of $\Gamma=\alpha\cup\beta$ then we obtain a 4-holed sphere $W_X\in\CW$. 

The points of $\CW$ are determined by the two lengths $\ell_\alpha$ and $\ell_\beta$. To determine points in $\CZ$ we need to furthermore specify the twists along $\alpha$ and $\beta$. It follows that $\CZ$ is the image of an analytic map from $\BR_{>0}^2\times\BR^2$ into the moduli space of hyperbolic structures on $\Sigma$. When we refer to a "generic" point in $\CZ$ we mean that it avoids the image of some lower-dimensional subset of $\BR_{>0}^2\times\BR^2$. 

For any generic $X\in\CZ$ we are going to construct a second hyperbolic structure $X'\in\CZ$ such that the pair $X,X'$ satisfies the conclusion of Example \ref{example separating pairs} for $\Gamma'$ still being the multicurve $\alpha\cup\beta$. Indeed, let $X'$ be the surface obtained from $X$ by cutting along $\alpha$ (not along $\Gamma$) and re-gluing with a half-length twist. In other words, if originally $x\in\D_\alpha W$ was glued to $y\in\D_\alpha' W$, now it is glued to its antipodal point $y'=\tau(y)\in\D_\alpha' W$. We stress that points in $\D_\beta W$ are identified in $X$ and $X'$ with the same points in $\D_\beta'W$.

\begin{lem}
    If $X$ is generic in $\CZ$, then $X$ and $X'$ are not isometric.
\end{lem}
\begin{proof}
    Since the mapping class group $\Map(X)$ is countable, and since $X\to X'$ is an analytic map with connected domain $\CZ$, it suffices to prove that there is some  $X\in\CZ$ for which $X'\notin\Map(X)\cdot X$. Now, if $X$ is such that $\alpha$ and $\tau(\alpha)$ is not only very short but also the unique shortest curve, and if there is a further simple closed geodesic $\gamma$ in $X$ such that $\gamma\cap W$ and $\tau(\gamma\cap W)$ are the unique shortest arcs from $\D_\alpha W$ to $\D_{\alpha'} W$, then $X$ is not isometric to $X'$ because any isometry $X\to X'$ would have to send $\alpha\subset X$ to $\alpha\subset X'$ but every curve in $X'$ which meets $\alpha$ is strictly longer than $\gamma\subset X$.
\end{proof}

Let from now on 
$$U=\CR(X)\cap T^1(W\setminus\D W)$$
be the set of random vectors $v\in\CR(X)\subset T^1X$ based at points in the interior $W=W_X$. For any such $v\in U$ and let ${\bf m}_{(X,\Gamma)}(v)=(\kappa_i^\Gamma(v),t_i^\Gamma(v))_{i\in\BZ}$ be its melody. Denote by 
$$I_i(v)=\rho_v^X\vert_{[t_i^\Gamma(v),t_{i+1}^\Gamma(v)]}$$ 
the restriction of the geodesics $\rho_v^X$ to the interval between the $i$-th and the $(i+1)$-st times it hits $\Gamma$. We think of the $I_i(v)$'s as a decomposition
$$\rho^X_v=\dots\cup I_{-2}(v)\cup I_{-1}(v)\cup I_0(v)\cup I_1(v)\cup I_2(v)\cup \dots$$
of $\rho_v^X$. We are going to use this decomposition to find a geodesic $\rho_{v'}^{X'}$ in $X'$ such that 
$${\bf m}_{(X,\Gamma)}(v)={\bf m}_{(X',\Gamma')}(v')$$
To do that consider 
$$\Delta(i)=\left\{
\begin{array}{ll}
\vert\{j\in\{0,\dots,i-1\}\text{ with }\kappa_i^\Gamma(v)=\alpha\}\vert\mod(2) & \text{ for }i\ge 0 \\
\vert\{j\in\{i,\dots,-1\}\text{ with }\kappa_i^\Gamma(v)=\alpha\}\vert\mod(2) & \text{ for }i< 0
\end{array}
\right.$$
for $i\in\BZ$. Let also $\dots,J_{-2}(v),J_{-1}(v),J_0(v),J_1(v),J_2(v),\dots$ be the bi-infinite sequence of parametrized geodesic segments in $W$ given by 
$$J_i(v)=\tau^{\Delta(i)}\circ I_i(v)\text{ for all }i.$$
The key observation is that when considered in $X$ the following holds: whenever one of those segments ends $J_i(v)$ in $\alpha$, then the subsequent segment $J_{i+1}(v)$ starts half-a-full-twist along $\alpha$, but when we consider them in $X'$ then their juxtaposition is a bi-infinite geodesic
$$\eta_v=\dots\cup J_{-2}(v)\cup J_{-1}(v)\cup J_0(v)\cup J_1(v)\cup J_2(v)\cup \dots$$
Now, since the segments $I_i(v)$ and $J_i(v)$ differ by the isometry $\tau$ of $W$, and since $\tau$ individually preserves the pairs $\{\D_\alpha W,\D_\alpha' W\}$ and $\{\D_\beta W,\D_\beta' W\}$, we get that the melody of $\eta_v$ is identical to that of $\rho^{X}_{v}$. It follows that
$$\Phi:U\to T^1X',\ v\mapsto\dot\eta_v(0)$$
sends vector to vectors with identical melodies. In more concrete terms, the map $\Phi$ can be described as follows: by construction both $X$ and $X'$ contain isometric copies of the interior of $W$. The map $\Phi$ is induced by the identification between $T^1(W\setminus\D W)\subset T^1X$ with $T^1(W\setminus\D W)\subset T^1X'$. From this last description we get that $\Phi$ is measurable, defined on a full-measure set, and not only absolutely continuous but actually measure preserving.

\subsection{Pants decompositions}
\begin{named}{Example \ref{example non-orientable}}
        For every $k\ge 2$ there are $3k-2$ pairwise non-isometric connected hyperbolic surfaces $X_0,\dots,X_{3k-3}$ with $\chi(X_i)=2-2k$, and for all $i=0,\dots,3k-3$ a pants decomposition $\Gamma_i\subset X_i$ such that the hyperbolic marimbas
    $$(X_0,\Gamma_0),\dots,(X_{3k-3},\Gamma_{3k-3})$$ 
    are isomelodic.
\end{named}

The construction is pretty similar to that in Example \ref{example separating pairs}. We start with a closed orientable hyperbolic surface $X_0$ of genus $k\ge 2$ and a pants decomposition $\Gamma_0$ with the following properties:
\begin{itemize}
    \item No component of $\Gamma_0$ separates $X_0$, and
    \item there is an orientation reversing isometry $\tau:X_0\to X_0$ which preserves each component $\gamma$ of $\Gamma_0$ individually ($\tau$ has exactly 2-fixed points on each $\gamma$).
\end{itemize}
Labeling the components of $\Gamma_0$ by $\gamma_1,\dots,\gamma_{3k-3}$ let $X_i$ be the (non-orientable) surface obtained by cutting $X_0$ along $\gamma_i$ and re-gluing via $\tau$ and letting $\Gamma_i=\{\hat\gamma_1,\dots,\hat\gamma_{3k-3}\}$ be the multicurve with 
\begin{itemize}
    \item for $j\neq i$, $\hat\gamma_j$ is the image of $\gamma_j$ under the embedding $X_0\setminus\gamma_i\subset X_i$, and
    \item $\hat\gamma_i\subset X_i$ is curve obtained from $\gamma_i\subset X_0$ but cutting and re-gluing.
\end{itemize}
Note that for all $i,j$ we have an identification $X_i\setminus\Gamma_i\simeq X_j\setminus\Gamma_j$. Let
$$\Phi^i_j:\CR(X_i)\cap T^1(X_i\setminus\Gamma_i)\to T^1(X_j\setminus\Gamma_j)$$
be the induced map at the level of unit tangent bundles. A similar argument as Example \ref{example separating pairs} shows that for all $i,j$ and all $v\in \CR(X_i,\Gamma_i)\cap T^1(X_i\setminus\Gamma_i)$ we have 
$${\bf m}_{(X_i,\Gamma_i)}(v)={\bf m}_{(X_j,\Gamma_j)}(\Phi^i_j(v)).$$
It is moreover easy to see that if the components of $\Gamma_0$ in $X_0$ are all short and of distinct lengths, then the surfaces $X_i$ and $X_j$ are not pair-wise isometric.

\begin{bem}
    In Example \ref{example non-orientable} we explained why we can get sets of cardinality $3k-2$ consisting of pairwise isomelodic elements. It is maybe worth noting that the same idea can be used to build much larger sets. In fact, instead of just cutting along a single component of $\Gamma_0$ and then re-gluing, one can cut along any subcollection in $\Gamma_0$ which fails to disconnect $X_0$. In this way one can build sets of exponentially growing cardinality consisting of pairwise isomelodic hyperbolic marimbas where the underlying surfaces have always $\chi(X_i)=2-2k$.
\end{bem}

Note that all but one of the surfaces in Example \ref{example non-orientable} are non-orientable. We do not know if there are two pants decompositions $\Gamma_1$ and $\Gamma_2$ of closed orientable hyperbolic surfaces $X_1$ and $X_2$ such that the marimbas $(X_1,\Gamma_1)$ and $(X_2,\Gamma_2)$ are isomelodic but not isometric. In other words, we do not know the answer to the following question:

\begin{quest}
    Suppose that $(X,\Gamma)$ and $(X',\Gamma')$ are orientable isomelodic non-isometric marimbas. Can we enlarge both $\Gamma$ and $\Gamma'$ to $\hat\Gamma$ and $\hat\Gamma'$, each by a single component, so that $(X,\hat\Gamma)$ and $(X',\hat\Gamma')$ are not isomelodic?
\end{quest}

\begin{quote}
    From now on, all surfaces are supposed to be orientable.
\end{quote}

\subsection{Different Euler characteristic}

\begin{named}{Example \ref{ex volume}}
    There are two isomelodic marimbas $(X,\Gamma)$ and $(X',\Gamma')$ with $\chi(X)\neq\chi(X')$.
\end{named}

Let $(X,\Gamma)$ be a hyperbolic marimba for which there is a homomorphism
$$\omega_n:\pi_1(X)\to \BZ/n\BZ$$
such that $\omega_n(\gamma)$ has order $n$ for every component $\gamma$ of $\Gamma$. Let $X'=\widetilde X/\Ker(\omega_n)$ be the cover of $X$ corresponding to the kernel of $\omega_n$, and note that the preimage $\pi^{-1}(\gamma)$ under the covering map $\pi:X'\to X$ of each component $\gamma$ of $\Gamma$ is connected. It follows that $\gamma\mapsto\pi^{-1}(\gamma)$ is a bijection between the sets of connected components of $\Gamma$ and of $\Gamma'=\pi^{-1}(\Gamma)$. Transporting to $\Gamma'$ the labeling of the components of $\Gamma$ via this bijection, we can think of $(X',\Gamma')$ as a hyperbolic marimba. 

Now, the basic, and basically trivial observation is that for any $v\in\CR(X')$ we have that
$${\bf m}_{(X,\Gamma)}(d\pi(v))={\bf m}_{(X',\Gamma')}(v),$$
meaning that two marimbas $(X,\Gamma)$ and $(X',\Gamma')$ are isomelodic. In fact, in one direction the map $\Phi$ provided by Theorem \ref{sat measure theory} is nothing other than the restriction to $\CR(X')$ of the differential $d\pi$ of the covering map. In the opposite direction, let $V\subset X$ be a simply connected subset of full measure and pick a smooth map $\sigma:V\to X'$ with $\pi\circ\sigma=\Id_U$. The restriction to $T^1V\cap\CR(X)$ of the differential $d\sigma$ of $\sigma$ is the map $\Phi$ provided by Theorem \ref{sat measure theory}.

\begin{bem}
    Let $\omega: \pi_1(X)\to \BZ$ be such that $\omega(\gamma)=\pm1$ for every component $\gamma$ of $\Gamma$, and define $\omega_n$ as the composition of $\omega$ and $p_n : \BZ \to \BZ/n\BZ$. Varying $n$, one obtains infinitely many isomelodic hyperbolic marimbas with pairwise distinct Euler characteristic.
\end{bem}

\subsection{Genus 2, separating curve and non-separating curve}

\begin{named}{Example \ref{example separating and non-separating}}
    There are two isomelodic hyperbolic marimbas $(X,\Gamma)$ and $(X,\Gamma')$ where $X,X'$ are non-isometric surfaces of genus $2$, where $\Gamma\subset X$ is a separating simple curve and $\Gamma'\subset X'$ is a non-separating one.
\end{named}

This example, like the previous one, is also based on taking covers. In this case let $\CO$ be a hyperbolic orbifold whose underlying surface is a torus and which has $2$ cone points, both with cone angle $\pi$. Let $\Gamma_\CO$ be a simple geodesic arc joining both endpoints. Now, the two surfaces $X$ and $X'$ will arise as the domains of two degree $2$ orbifold covers
$$\pi:X\to\CO\text{ and }\pi':X'\to\CO$$
and we will have $\Gamma=\pi^{-1}(\Gamma_\CO)$ and $\Gamma'=(\pi')^{-1}(\Gamma_\CO)$. Since the covers have degree $2$ and since the domains are actual surfaces and not merely orbifolds, we get that in that setting both $\Gamma$ and $\Gamma'$ will consist of a single component. Moreover, for any two $v\in\CR(X)$ and $v'\in\CR(X')$ with $d\pi(v)=d\pi'(v')$ we have
$${\bf m}_{(X,\Gamma)}(v)={\bf m}_{(X',\Gamma')}(v'),$$
meaning that the two hyperbolic marimbas $(X,\Gamma)$ and $(X',\Gamma')$ are isomelodic. It follows that what we have to do is to build the two degree 2 covers $\pi:X\to\CO$ and $\pi':X'\to\CO$ so that $\Gamma$ is separating and $\Gamma'$ isn't. 

Since thinking of orbifolds is a bit complicated let us denote by $\CO^*$ the topological surface obtained by deleting the two cone points of $\CO$. We can thus think of $\Gamma_\CO$ as a properly embedded arc $\Gamma^*\subset\CO^*$. Consider now the cover $X^*\to\CO^*$ corresponding to the kernel of the cohomology class
$$\pi_1(\CO^*)\to\BZ/2\BZ,\ \alpha\mapsto\iota(\alpha,\Gamma^*)\mod(2).$$
When we denote by $X$ the surface obtained by closing up both cusps of $X^*$, then the cover $\pi:X^*\to\CO^*$ extends to an orbifold cover $\pi:X\to\CO$ such that $\Gamma=\pi^{-1}(\Gamma_\CO)$ is a separating curve. 

To construct the second cover, let $\eta\subset\CO^*$ be another properly embedded arc going from a puncture to the other, disjoint and not properly homotopic to $\Gamma_*$. The cohomology class 
$$\pi_1(\CO^*)\to\BZ/2\BZ,\ \alpha\mapsto\iota(\alpha,\eta)\mod(2)$$
determines again a degree $2$ cover $(X')^*\to\CO^*$, which again extends to a degree $2$ orbifold cover $\pi':X'\to\CO$. This time the preimage $\Gamma'=(\pi')^{-1}(\Gamma_\CO)$ is a non-separating curve. 

\subsection{A remark and a question} Although this is not how we presented all of them, all those examples can be obtained as follows. Let $\CO=G_\CO\bs\BH^2$ be a hyperbolic orbifold whose singular points, if any, have cone angle $\pi$. Let $\Gamma_\CO\subset\CO$ be a labeled collection of simple disjoint geodesic arcs and curves, where arcs means that their endpoints are distinct cone-points, and where disjoint means that they are disjoint in the underlying topological surface. Every $\BS^1$ component $\gamma$ of $\Gamma_\CO$ determines a maximal cyclic subgroup $\langle\gamma\rangle$ of $G_\CO$. If $\gamma\subset\Gamma_\CO$ is an arc, then it determines an infinite dihedral group which we again denote by $\langle\gamma\rangle$.

We will say that a finite index subgroup $G\subset G_\CO$ satisfies condition (M) if for every $\gamma\subset\Gamma_\CO$ we have
\begin{equation}\tag{M}
    \vert\langle\gamma\rangle/(\langle\gamma\rangle\cap G)\vert=\vert G_\CO/G\vert
\end{equation}
Now, if a torsion-free $G$ subgroup satisfies condition (M), then the preimage of every component of $\Gamma_\CO$ under the orbifold cover 
$$X=G\bs\BH^2\to G_\CO\bs\BH^2=\CO$$
is connected. It follows that the labeling of $\Gamma_\CO$ induces a labeling of $\Gamma=\pi^{-1}(\Gamma_\CO)$. We can thus think of $(X,\Gamma)$ as a hyperbolic marimba, the {\em hyperbolic marimba associated to $G$}. 

Now, the basic observation is that {\em if $G, G' \subset G_\CO$ satisfy both condition (M), then the associated hyperbolic marimbas $(X,\Gamma)$ and $(X',\Gamma')$ are isomelodic.}

As we mentioned, all examples above can be recovered in this way. In fact, we have no idea if other examples are possible or not:

\begin{quest}\label{question2}
    Are there two isomelodic hyperbolic marimbas $(X,\Gamma)$ and $(X',\Gamma')$ with the property that the two surfaces $X$ and $X'$ are not commensurable?
\end{quest}

The construction we described above is eerily reminiscent of the Sunada construction of isospectral hyperbolic surfaces \cite{Sunada}. In that setting, it seems to be unknown if there are non-commensurable examples (that would be non-arithmetic by work of Reid~\cite{Reid_commesurable_arithmetic_surfaces}).

\section{A little bit of Teichm\"uller theory}\label{sec teichmueller theory}
In the sequel we will need some basic facts about Teichm\"uller space. We discuss this now. We will be assuming that the reader knows Teichm\"uller space, Fenchel-Nielsen coordinates, and that these coordinates yield a real analytic structure on Teichm\"uller space. We refer to \cite{Buser} for details on all those matters.

\subsection*{Teichm\"uller space of a marimba}
We suppose throughout this section that $(X_0,\Gamma_0)$ is an oriented topological marimba and choose once and forever an orientation of the components of $\Gamma_0$.

Under the {\em Teichm\"uller space} $\CT(X_0,\Gamma_0)$ of the marimba $(X_0,\Gamma_0)$ we understand the set equivalence classes of marked hyperbolic marimbas $\phi:(X_0,\Gamma_0)\to(X,\Gamma)$ with $\phi$ an orientation preserving homeomorphism, and where, as always, two such $\phi:(X_0,\Gamma_0)\to(X,\Gamma)$ and $\phi':(X_0,\Gamma_0)\to(X',\Gamma')$ are equivalent if there is an (label preserving) isometry $\psi:(X,\Gamma)\to(X',\Gamma')$ such that $\phi'$ and $\psi\circ\phi$ are homotopic. Note that there is not much difference between the Teichm\"uller spaces $\CT(X_0,\Gamma_0)$ and $\CT(X_0)$ of the marimba $(X_0,\Gamma_0)$ and of the surface $X_0$ itself. Indeed, the map $\CT(X_0)\to\CT(X_0,\Gamma_0)$ sending the class of $\phi:X_0\to X$ to that of $\phi:(X_0,\Gamma_0)\to(X,\phi(\Gamma_0))$ is a well-defined bijection. 

\subsection*{Arcs and their length functions}
Under an {\em arc} in a topological marimba $(X_0,\Gamma_0)$ we will understand a map $\alpha:([0,1],\{0,1\})\to(X_0,\Gamma_0)$. Two arcs are homotopic if they are homotopic as maps of pairs. An arc is {\em essential} if it is not homotopic to a constant map. Arcs in $(X_0,\Gamma_0)$ will always be assumed to be essential and transversal to $\Gamma_0$, with a minimal number of intersections in their homotopy classes. 

\begin{defi}
    A {\em $k$-step arc} in a marimba $(X_0,\Gamma_0)$ is an arc $\alpha:([0,1],\{0,1\})\to(X_0,\Gamma_0)$ with $k+1=\vert[0,1]\cap\alpha^{-1}(\Gamma)\vert$.
\end{defi}

For example, an arc $\alpha\subset(X_0,\Gamma_0)$ whose interior does not meet $\Gamma_0$ is a $1$-step arc. We evidently have a bijection between $1$-step arcs in $(X_0,\Gamma_0)$ and proper arcs in the metric completion $Y_0$ of $X_0\setminus\Gamma_0$.
\smallskip

Suppose now that $\alpha$ is an arc in $(X_0,\Gamma_0)$. For $X=(X,\Gamma)\in\CT(X_0,\Gamma_0)$ we can identify via the marking the homotopy class of $\alpha$ in $(X_0,\Gamma_0)$ with the homotopy class of an arc, which we still denote by $\alpha$, in $X$. This homotopy class has a preferred geodesic representative $\alpha_X$, namely the one which at the beginning and the end is orthogonal to $\Gamma$--we say that such arcs are {\em orthogeodesic} and we refer to the length of the orthogeodesic representative as the {\em length} $\ell_X(\alpha)$ of the homotopy class. In this way we get a map
$$\CT(X_0,\Gamma_0)\to\BR_{>0},\ X\mapsto\ell_X(\alpha),$$
the {\em length function of the arc $\alpha$}.
\smallskip

If $\alpha$ is an arc in $(X_0,\Gamma_0)$ is the juxtaposition of two arcs $\eta^1$ and $\eta^2$ with endpoints on $\Gamma_0$, then for any  $X=(X,\Gamma)\in\CT(X_0,\Gamma_0)$ there is some uniquely determined $d_{\alpha,\eta^1,\eta^2}(X)\in\BR$ such that the homotopy class $\alpha$ has a unique representative in $(X,\Gamma)$ of the following form: first run the orthogeodesic representative $\eta^1_X$, then run along $\Gamma$ for (signed) length $d_{\alpha,\eta^1,\eta^2}(X)$, and finally run the orthogeodesic segment $\eta^2_X$. We can calculate $\ell_X(\alpha)$ from $\ell_X(\eta^1_X)$, $\ell_X(\eta^2_X)$ and $d_{X}(\alpha)$ via the formula:
\begin{equation}\label{eq thm 244 buser}
\begin{split}
\cosh(\ell_X(\alpha))=&\sinh(\ell_X(\eta^1))\cdot\sinh(\ell_X(\eta^2))\cdot\cosh\big(d_{\alpha,\eta^1,\eta^2}(X)\big)\\
&+\cosh(\ell_X(\eta^1))\cdot\cosh(\ell_X(\eta_2)).   
\end{split}
\end{equation}
See \cite[Thm. 2.4.4]{Buser}.

\begin{figure}[h]
\leavevmode \SetLabels
\L(.49*.92) $\alpha$\\%
\L(.33*1.02) $\eta^1$\\%
\L(.50*-.1) $\eta^2$\\%
\L(.35*.10) $I$\\%
\endSetLabels
\begin{center}
\AffixLabels{\centerline{\includegraphics[width=0.4\textwidth]{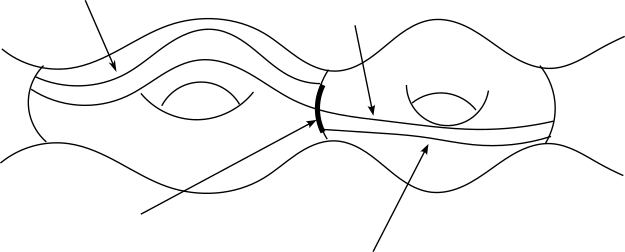}}\hspace{1cm}}
\end{center}
\caption{The bold printed segment $I$ has length $d_{\alpha,\eta^1,\eta^2}(X)$.}
\end{figure}
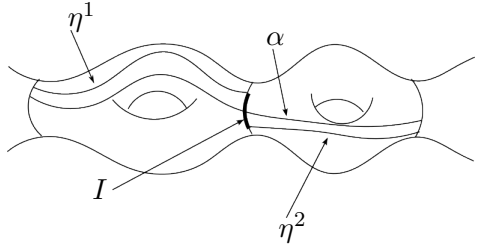

\subsection*{Analytic structure}
As we already did before, let $Y_0$ be the metric completion of $X_0\setminus\Gamma_0$. Abusing notation we will often not distinguish between $Y_0$ and $X_0\setminus\Gamma_0$ itself. Anyways, to recover $X_0$ from $Y_0$ we glue together pairs of boundary components of $Y_0$ via orientation reversing homeomorphisms. Each such pair $\D^-_\gamma Y_0,\D^+_\gamma Y_0$ corresponds to a component $\gamma$ of $\Gamma$. For the sake of concreteness we assume that the identification $\D^-_\gamma Y\simeq\gamma$ is orientation preserving while $\D^+_\gamma Y\simeq\gamma$ is orientation reversing. Here we have endowed $Y_0$ with the orientation it inherits from $X_0$, and oriented $\D Y_0$ in such a way that when we march in positive direction then the interior is on the right.

Denote by $\CT_b(Y_0)$ the subset of the Teichm\"uller space $\CT(Y_0)$ consisting of those points $Y$ for which the boundary components $\D^+_\gamma Y_0$ and $\D^-_\gamma Y_0$ have the same length for every component $\gamma$ of $\Gamma_0$:
$$\CT_b(Y_0)=\{Y\in\CT(Y_0)\text{ with }\ell_Y(\D_\gamma^+Y_0)=\ell_Y(\D_\gamma^-Y_0)\text{ for all }\gamma\in\Gamma_0\}$$
Extending $\Gamma_0$ to a pants decomposition $P$ of $X_0$ we get also one of $Y_0$. Fenchel-Nielsen coordinates yield then a homeomorphism
$$\Phi:\CT_b(Y_0)\times\BR^{\Gamma_0}\to\CT(X_0,\Gamma_0).$$
There are plenty of choices for what one could call "Fenchel-Nielsen coordinates" \cite{BBFS}, but we will choose them as in \cite{Buser}.

Teichm\"uller space has a canonical real analytic structure with respect to which it is analytically equivalent to Euclidean space:
$$\CT(X_0,\Gamma_0)\simeq\BR^{3\vert\chi(X_0)\vert}\text{ and }\CT_b(Y_0)\simeq\BR^{3\vert\chi(X_0)\vert-\vert\Gamma_0\vert}.$$
Moreover, both the map $\Phi$ and its inverse are analytic. It follows also easily from the construction of Fenchel-Nielsen coordinates and \eqref{eq thm 244 buser} that the functions $X\mapsto\ell_\alpha(X)$ and $X\mapsto d_{\alpha,\eta^1,\eta^2}(X)$ are analytic on $\CT(X_0,\gamma)$.

We summarize in the following proposition a few facts we will need later on.

\begin{prop}\label{prop FN}
    Let $(X_0,\Gamma_0)$ be a closed oriented topological marimba:
    \begin{itemize}
        \item The functions $X\mapsto\ell_\alpha(X)$ and $X\mapsto d_{\alpha,\eta^1,\eta^2}(X)$ are analytic on $\CT(X_0,\gamma)$.
        \item If we denote by $Y_0$ be the metric completion of $X_0\setminus\Gamma_0$, then there is an analytic isomorphism
    $$\Phi:\CT_b(Y_0)\times\BR^{\Gamma_0}\to\CT(X_0,\Gamma_0)$$
    such that $\Phi(Y,(\theta_\gamma)_{\gamma\in\Gamma_0})$ is the hyperbolic marimba obtained by identifying, for all $\gamma\in\Gamma_0$, the boundary components $\D_\gamma^-Y$ and $\D_\gamma^+Y$ of $Y$ via an orientation reversing isometry $\sigma_{Y,\gamma}^{\theta_\gamma}:\D_\gamma^-Y\to\D_\gamma^+Y$. 
    \item Moreover, for different choices $(\theta_\gamma),(\theta'_\gamma)\in\BR^{\Gamma_0}$, the isometries $\sigma_{Y,\gamma}^{\theta_\gamma}$ and $\sigma_{Y,\gamma}^{\theta'_\gamma}$ are related by
        $$\sigma_{Y,\gamma}^{\theta_\gamma}(x)=\sigma_{Y,\gamma}^{\theta'_\gamma}(x+\theta_\gamma-\theta'_\gamma)$$
    where $x+\theta_\gamma-\theta_\gamma'$ is the point in $\D^-_\gamma Y$ that one reaches from $x$ by running for time one with signed speed $\theta_\gamma-\theta'_\gamma$.
\qed
    \end{itemize}
\end{prop}

For the sake of concreteness, suppose now that $(X_0,\Gamma_0)$ is a single-note marimba, meaning that $\Gamma_0$ consists of a single curve. The homeomorphism $\Phi$ provided by Proposition \ref{prop FN}, being obtained from via Fenchel-Nielsen coordinates, has the property that the points $\Phi(Y,\theta)$ and $\Phi(Y,\theta+k\cdot\ell_Y(\D^-Y))$ differ just in their marking. Said differently, they represent isometric hyperbolic marimbas. Suppose now that the curve $\Gamma$ separates $X$ and let $Y',Y''$ be the two components of $Y=X\setminus\Gamma$. If there is an orientation preserving isometry $\phi:Y'\to Y'$ which acts on the boundary as an isometry of order $p$, then the points $\Phi(Y,\theta)$ and $\Phi(Y,\theta+\frac 1p\ell_Y(\D^-Y))$ also represent isometric hyperbolic marimbas. This is particularly important when $Y'$ is a torus: every one-holed torus has a non-trivial isometry, namely the hyper-elliptic involution. We record these facts:

\begin{named}{Addendum to Proposition \ref{prop FN}}
    Suppose that $\Gamma_0$ consists of a single curve. The following holds for all $(Y,\theta)\in\CT_b(Y_0)\times\BR$:
    \begin{itemize}
        \item The hyperbolic marimbas $\Phi(Y,\theta),\Phi(Y,\theta+k\cdot\ell_Y(\D^-Y))\in\CT(X_0,\Gamma_0)$ are isometric for all $k\in\BZ$.
        \item If $\Gamma_0$ separates $X_0$ and (at least) one of the components of $Y_0$ is a torus, then the hyperbolic marimbas $\Phi(Y,\theta),\Phi(Y,\theta+\frac 12\ell_Y(\D^-Y))$ are isometric.\qed
    \end{itemize}
\end{named}

\subsection*{A technical lemma}
Later on we will be interested in maps of the form
$$\CT_b(Y_0)\times\BR^\Gamma\to\BR^k,\ \ (Y,\theta)\mapsto(\ell_{\Phi(Y,\theta)}(\tau_i))_{i=1,\dots,k}$$
where $\tau_1,\dots,\tau_k$ are arcs in $(X,\Gamma)$, and where we $\Phi$ is as in Proposition \ref{prop FN}. Being the composition of analytics maps, such maps are analytic. However, the images of analytic maps are in general not analytic sets. This is the reason why we have to work a bit to prove the following lemma:

\begin{lem}\label{lem analytic function}
    With notation as above, suppose that $\Gamma_0$ consists of a single curve and for some $r\ge 2$ let $\tau_1,\dots,\tau_r$ geodesic 2-step arcs in $(X_0,\Gamma_0)$, that is that they only meet $\Gamma_0$ once in their interior. As above set $Y_0=X_0\setminus\Gamma_0$. There is an analytic map 
    $$P:\CT_b(Y_0)\times\BR_{>0}^{r}\to\BR^{r-1}$$
    such that 
    $$\left\{(Y,\cosh(\ell_{\Phi(Y,\theta)}(\tau_1)),\dots,\cosh(\ell_{\Phi(Y,\theta)}(\tau_r)))\in\CT_b(Y_0)\times\BR_{>0}^r\ \middle\vert\  \theta\in\BR
    \right\}=\left\{P=0\right\}$$
    where $\Phi$ is as in Proposition \ref{prop FN}.
\end{lem}

\begin{bem}
It is worth mentioning that Lemma \ref{lem analytic function} fails for $r=1$ because the set 
$$\{(Y,\cosh(\ell_{\Phi(Y,\theta)}(\tau)))\in\CT_b(Y_0)\times\BR_{>0}\ \vert\  \theta\in\BR\}$$ 
is a proper closed subset of $\CT_b(Y_0)\times\BR_{>0}$ with non-empty interior.    
\end{bem}

\begin{proof}
    For $i=1,\dots,r$ let $\alpha_i$ and $\beta_i$ be the two arcs we get from $\tau_i$ by cutting at the middle intersection with $\Gamma_0$, and set $d_i(Y)=d_{\tau_i,\alpha_i,\beta_i}(\Phi(Y,0))$ as in Proposition \ref{prop FN}. Since $\tau_i$ is a 2-step arc, we can think of $\alpha_i,\beta_i$ as proper arcs in $Y$. We get from Proposition \ref{prop FN} that the functions
    $$Y\mapsto\ell_Y(\alpha_i),\ Y\mapsto\ell_Y(\beta_i),\text{ and }Y\mapsto d_i(Y)$$
    on $\CT_b(Y_o)$ are all analytic. Moreover, we get from \eqref{eq thm 244 buser} that 
    $$\cosh(\ell_{\Phi(Y,\theta)}(\tau_i))=\sinh(\ell_Y(\alpha_i))\sinh(\ell_Y(\beta_i))\cosh( d_i(Y)+\theta)+\cosh(\ell_Y(\alpha_i))\cosh(\ell_Y(\beta_i))$$
    To relax notation write $(Y,\theta)$ instead of $\Phi(Y,\theta)$ and set 
    $$a_i(Y)=\sinh(\ell_Y(\alpha_i))\sinh(\ell_Y(\beta_i))\text{ and }b_i(Y)=\cosh(\ell_Y(\alpha_i))\cosh(\ell_Y(\beta_i)).$$
    With this notation the formula above reads
    $$\cosh(\ell_{(Y,\theta)}(\tau_i))=a_i(Y)\cosh(d_i(Y)+\theta)+b_i(Y)$$
    Invoking the addition theorem for cosh and writing
    $$\delta_i(Y)=d_i(Y)-d_1(Y)$$
    the previous formula becomes
    $$\cosh(\ell_{(Y,\theta)}(\tau_i))=a_i(Y)\big(\cosh(\delta_i(Y))\cosh(d_1(Y)+\theta)+\sinh(\delta_i(Y))\sinh( d_1(Y)+\theta)\big)+b_i(Y)$$
    When we write them all out for $i=1,\dots,r$ we get
\begin{align*}
    \cosh(\ell_{(Y,\theta)}(\tau_1))&=a_1(Y)\cosh(d_1(Y)+\theta)+b_1(Y)\\
    \cosh(\ell_{(Y,\theta)}(\tau_2))&=a_2(Y)\big(\cosh(\delta_2(Y))\cosh(d_1(Y)+\theta)+\sinh(\delta_2(Y))\sinh( d_1(Y)+\theta)\big)+b_2(Y)\\
    \dots &=\dots\\
    \cosh(\ell_{(Y,\theta)}(\tau_r))&=a_r(Y)\big(\cosh(\delta_r(Y))\cosh(d_1(Y)+\theta)+\sinh(\delta_r(Y))\sinh( d_1(Y)+\theta)\big)+b_r(Y)
\end{align*}
Observing that $a_i(Y)>0$, consider the affine map
$$A_Y:\BR^r\to\BR^r,\ \left(\begin{array}{c}x_1\\ x_2\\ \ldots \\ x_r \end{array}\right)\mapsto\left(\begin{array}{c}\frac 1{a_1(Y)}(x_1-b_1(Y)) \\ \frac 1{a_2(Y)\cosh(\delta_2(Y))}(x_2-b_2(Y))\\ \dots \\ \frac 1{ a_r(Y)\cosh(\delta_r(Y))}(x_r-b_r(Y)) \end{array}\right).$$
Applying $A_Y$ to the column vector with entries $\cosh(\ell_{(Y,\theta)}(\tau_i))$ for $i=1,\dots,r$ one gets the following:
$$A_Y\left(\begin{array}{c}\cosh(\ell_{(Y,\theta)}(\tau_1))\\ \cosh(\ell_{(Y,\theta)}(\tau_2))\\ \ldots \\ \cosh(\ell_{(Y,\theta)}(\tau_r)) \end{array}\right)=\left(\begin{array}{c}\cosh(d_1(Y)+\theta)\\ \cosh( d_1(Y)+\theta)+\tanh(\delta_2(Y))\sinh(d_1(Y)+\theta)\\ \ldots \\ \cosh( d_1(Y)+\theta)+\tanh(\delta_r(Y))\sinh(d_1(Y)+\theta) \end{array}\right)
$$
Now, when we vary $\theta$, the set of vectors on the right is nothing other than the vanishing set of the function
$$F_Y:\BR_{>0}\times\BR^r\to\BR^{r-1},\ F_Y:\left(\begin{array}{c}x_1 \\ x_2 \\ \vdots \\ x_r \end{array}\right)\mapsto\left(\begin{array}{c}\tanh(\delta_2(Y))^2\cdot x_1^2-(x_2-x_1)^2-\tanh(\delta_2(Y))^2 \\ \vdots \\ \tanh(\delta_r(Y))^2\cdot x_1^2-(x_r-x_1)^2-\tanh(\delta_r(Y))^2 \end{array}\right)$$
In other words, we have proved that 
\begin{multline*}
    \left\{(Y,(\cosh(\ell_{(Y,\theta)}(\tau_i)))_{i=1,\dots,k})\middle\vert Y\in\CT_b(Y_0),\ \theta\in\BR\right\}=\\
    =\left\{(Y,(x_i)_i)\in\CT_b(Y_0)\times\BR_{>0}^r\middle\vert F_Y\left(A_Y\left((x_i)_i\right)\right)=0\right\}  
\end{multline*}
Since entries of $A_Y$ and $F_Y$ depend analytically on $Y$, the claim follows when we set
$$P(Y,x_1,\dots,x_r)=F_Y\left(A_Y(x_i)_{i=1,\dots,r}\right).$$        
\end{proof}

\section{$k$-step orthospectrum of $(X,\Gamma)$}\label{sec hear orthospectrum}

We remind the reader that we always assume that arcs $\alpha$ in a hyperbolic marimba $(X,\Gamma)$ are essential and meet $\Gamma$ as few times as possible in their homotopy classes. Keep also in mind that the length $\ell_X(\alpha)$ of the homotopy class of $\alpha$ is the length of its unique orthogeodesic representative $\alpha_X$. Finally, recall that a $k$-step arc $\alpha:([0,1],\{0,1\})\to(X,\Gamma)$ is an arc with
$$k+1=\vert[0,1]\cap\alpha^{-1}(\Gamma)\vert.$$
In this section we will be interested in the multiset $\CG_k(X,\Gamma)$ of lengths of all homotopy classes of $k$-step arcs. Here "multiset" refers to the fact that if there are $r$ arcs of length $l$, then $l$ appears $r$ times in $\CG_k(X,\Gamma)$. Anyways, we refer to $\CG_k(X,\Gamma)$ as the {\em $k$-step orthospectrum} of the hyperbolic marimba $(X,\Gamma)$. 

Our goal is to show that the melody ${\bf m}_{(X,\Gamma)}(v)$ of a random $v\in T^1X$ determines the $k$-step orthospectrum of $(X,\gamma)$ for all $k$. In light of the examples in Section \ref{sec examples} we have to be slightly careful with the statement:

\begin{prop}\label{prop hear orthospectrum}
    Let ${\bf m}\in(\Gamma\times\BR)^\BZ$ be such that there are a hyperbolic marimba $(X,\Gamma)$ and a random vector $v\in\CR(X)$ with ${\bf m}={\bf m}_{(X,\Gamma)}(v)$. If we know $\chi(X)$, then ${\bf m}$ determines the $k$-step orthospectrum $\CG_k(X,\gamma)$ for all $k\ge 1$.
\end{prop}


Before launching the proof of Proposition \ref{prop hear orthospectrum}, let us comment briefly on the etymology of the word {\em $k$-step orthospectrum} of a hyperbolic marimba $(X,\Gamma)$. The reason for this name is namely that when $Y$ is a compact surface with totally geodesic boundary, then the set (with multiplicity) of lengths of orthogeodesic arcs in $Y$ is the {\em orthospectrum $\CG(Y)$ of $Y$}. Formally, these two notions are related as follows. Let $Y$ be the metric completion of $X\setminus\Gamma$. Every $1$-step orthogeodesic arc in $(X,\Gamma)$ determines uniquely an orthogeodesic arc in $Y$, and vice versa, meaning that we have
$$\CG(Y)=\CG_1(X,\Gamma).$$
In particular we get directly from Proposition \ref{prop hear orthospectrum} the following fact that we state here for further reference:

\begin{kor}\label{kor hear orthospectrum}
    Let ${\bf m}\in(\Gamma\times\BR)^\BZ$ be such that there are a hyperbolic marimba $(X,\Gamma)$ and a random vector $v\in\CR(X)$ with ${\bf m}={\bf m}_{(X,\Gamma)}(v)$. If we know $\chi(X)$, then ${\bf m}$ determines the orthospectrum $\CG(Y)$ of the metric completion $Y$ of $X\setminus\Gamma$.\qed
\end{kor}


Starting with the proof of Proposition \ref{prop hear orthospectrum}, fix $k\ge 1$, and suppose from now on that $(X,\Gamma)$ is a hyperbolic marimba of which we know $\chi(X)$, and that $v\in\CR(X)$ is such that
$${\bf m}={\bf m}_{(X,\Gamma)}(v)=(\kappa_i^\Gamma,t_i^\Gamma)_i.$$
To simplify notation, we will from now on write $(t_i)_i=(t_i^\Gamma)_i$. We will prove that $\chi(X)$ and the sequence $(t_i)_i$ determine $\CG_k(X,\Gamma)$. 

The first step is to prove that the probability measures
$$\nu_n=\frac 1n\sum_{i=1}^n\delta_{\rho_v(t_i)}$$
on $T^1X\vert_\Gamma$ converge as $n\to\infty$ to the probability measure $\lambda$ satisfying
    $$\int_{T^1X\vert_\Gamma}f(v)\ d\lambda(v)=\frac 1{4\cdot \ell_X(\Gamma)}\int_\Gamma\int_{\BS^1}f(x,\theta)\vert\sin(\theta)\vert\ d\theta dx$$
for every measurable function $f$ on $T^1X\vert_\Gamma$. Here, $\delta_z$ is the Dirac measure centered at $z$, and convergence means that
$$\lim_n\int_{T^1X\vert_\Gamma} f\ d\nu_n=\int_{T^1X\vert_\Gamma}f\ d\lambda$$
for every continuous function $f\in C^0(T^1X\vert_\Gamma)$.

\begin{lem}\label{lem hit distribution notes}
    As $n\to\infty$, the measures $\nu_n$ converge to the probability measure $\lambda$.
\end{lem}

\begin{proof}
    Noting that $\lambda$ and the standard Lebesgue measure on $T^1X\vert_\Gamma$ are in the same measure class, we get from Lemma \ref{lem distribution of vectors over time} that for every $V\subset T^1X\vert_\Gamma^*$ open with $\lambda(\overline V\setminus V)=0$ we have 
    \begin{align*}
    \frac 1{4\pi^2\vert\chi(X)\vert}\int_V\vert\sin(\theta)\vert\ d\theta dx
    &=\lim_{T\to\infty}\frac 1T\vert\{t\in[0,T]\text{ with }\rho_v(t)\in V\}\vert\\
    &=\lim_{T\to\infty}\frac 1T\left(\sum_{1\le i,\ t_i\le T}\delta_{\rho_v(t_i)}\right)(V).  \end{align*}
    The classic Portmanteau theorem implies thus that
    \begin{equation}\label{eq I am drinking a beer}
        \lim_{T\to\infty}\frac 1T\sum_{1\le i,\ t_i\le T}\delta_{\rho_v(t_i)}=\frac 1{4\pi^2\vert\chi(X)\vert}\cdot\int\vert\sin(\theta)\vert\ d\theta dx.
    \end{equation}
    Note now that $N_\Gamma(v,T)$ as in \eqref{eq note counting} is nothing other than the number of notes played by time $T$, that is the number of summands in the sum within the limit in \eqref{eq I am drinking a beer}. Recall also that by \eqref{eq note counting proportional to time} we have
    $$N_\Gamma(v,T)\sim \frac 1{\pi^2\vert\chi(X)\vert}\ell_X(\Gamma)\cdot T$$  
    where $\sim$ means that the ratio between both sides tends to $1$ as $T\to\infty$. In particular, solving for $T\sim \frac{\pi^2\vert\chi(X)\vert}{\ell_X(\Gamma)}\cdot N_\Gamma(v,T)$ we can rewrite the previous equality as follows:
    $$\lim_{T\to\infty}\frac{\ell_X(\Gamma)}{\pi^2\vert\chi(X)\vert\cdot N_\Gamma(v,T)}
    \sum_{1\le i\le N_\Gamma(v,T)}\delta_{\rho_v(t_i)}=\frac 1{4\pi^2\vert\chi(X)\vert}\cdot \int\vert\sin(\theta)\vert\ d\theta dx.$$
    This can be finally rewritten as 
    $$\lim_{T\to\infty}\frac{1}{N_\Gamma(v,T)}
    \sum_{1\le i\le N_\Gamma(v,T)}\delta_{\rho_v(t_i)}=\frac 1{4\cdot\ell_X(\Gamma)}\cdot \int \vert\sin(\theta)\vert\ d\theta dx.$$
    The claim follows when we notice that we can write this as 
    $$\lim_{n\to\infty}\frac{1}{n}
    \sum_{1\le i\le n}\delta_{\rho_v(t_i)}=\lambda.$$
    We have proved Lemma \ref{lem hit distribution notes}.
\end{proof}

Let's consider now the function
\begin{equation}\label{eq random variable tau}
    \begin{split}
    \tau&:T^1X\vert_\Gamma\to[0,\infty],\\
    \tau&:v\mapsto\inf\left\{t>0\ \middle\vert\ \begin{array}{l}
    \text{there are }0=t_0<t_1<\dots<t_k<t_{t+1}=t\\ \text{with }\rho_v^X(t_i)\in T^1X\vert_\Gamma\text{ for }i=0,\dots,k+1\end{array}\right\}
    \end{split}
\end{equation}
which we will think of as a random variable on the probability spaces $(T^1X\vert_\Gamma,\lambda)$ and $(T^1X\vert_\Gamma,\nu_n)$. We denote the corresponding cumulative distribution functions by 
$$\phi_\lambda(T)=\lambda(\tau^{-1}(0,T])\text{ and }\phi_n(T)=\nu_n(\tau^{-1}(0,T]).$$
More generally we will also consider the cumulative distribution functions
$$\phi_\lambda^U(T)=\lambda(U\cap \tau^{-1}(0,T])\text{ and }\phi_n^U(T)=\nu_n(U\cap \tau^{-1}(0,T])$$
of the restriction of $\tau$ to suitable subsets $U\subset T^1X$.

\begin{lem}\label{lem cumulative distribution functions converge}
    Let $U\subset T^1X\vert_\Gamma$ be an open set with $\lambda(\D U)=0$. The cumulative distribution function $t\mapsto\phi^U_\lambda(t)$ is continuous and we have
    $$\lim_{n\to\infty}\phi^U_n(t)=\phi^U_\lambda(t)$$
    uniformly on compacta in $\BR_{\ge 0}$.
\end{lem}
\begin{proof}
    The set $\tau^{-1}(0,\infty)$ is open and has full Lebesgue measure in $T^1X\vert_\Gamma$. It has hence full $\lambda$-measure. Moreover, restricted on that set $\tau$ is analytic and nowhere locally constant. It follows that the point preimages $\tau^{-1}(t)$ have vanishing Lebesgue measure, and hence vanishing $\lambda$-measure. This means that the push-forward measure $(\tau\vert_U)_*\lambda$ is not atomic and hence that the cumulative distribution function $\phi^U_\lambda$ is continuous. For the same reason we have that 
    $$\lambda(\D((\tau\vert_U)^{-1}(a,b)))\le\lambda(\D U)+\lambda(\tau^{-1}(\{a,b\})=0\text{ for all }(a,b)\subset\BR_{>0}.$$
    Since the measures $\nu_n$ converge to $\lambda$ by Lemma \ref{lem hit distribution notes}, we get from the Portmanteau theorem that
    $$\lim_n\nu_n((\tau\vert_U)^{-1}(a,b))=\lambda((\tau\vert_U)^{-1}(a,b))$$
    In terms of the cumulative distribution functions we can rewrite this as
    $$\lim_n(\phi^U_n(b)-\phi^U_n(a))=\phi^U_\lambda(b)-\phi^U_\lambda(a)$$
    Since $\phi^U_\lambda$ is continuous, since there is some $\epsilon>0$ with $\phi^U_\lambda(t)=\phi^U_n(t)=0$ for all $t\in[0,\epsilon)$ and for all $n$, and since the functions $\phi^U_\lambda$ and $\phi^U_n$ are all monotonic, it follows that $\phi^U_n$ converges to $\phi^U_\lambda$ uniformly on compacta, as claimed.
\end{proof}

To every $v\in T^1X\vert_\Gamma$ with $0<\tau(v)<\infty$ we can associate the $k$-step geodesic arc
$$I_v:[0,1]\to X,\ t\mapsto\rho_v(t\cdot\tau(v))$$
with end points in $\Gamma$. Now, if $w\in T^1X\vert_\Gamma$ is sufficiently close to $v$, then not only do we have $\tau(w)<\infty$ but also that $I_w$ is homotopic to $I_v$. In other words, the homotopy class of the $k$-step arc $I_v$ is locally constant on the open set $\tau^{-1}(0,\infty)$. Moreover, since every arc $I_v$ can be homotoped to an orthogeodesic arc, and since this arc is unique in its homotopy class, we get that indeed the connected components of $\tau^{-1}(0,\infty)$ are in one-to-one correspondence to the homotopy classes of $k$-step arcs in $X\setminus\Gamma$--note that our arcs are oriented.

We need some notation. If $I$ is a $k$-step orthogeodesic arc in $(X,\Gamma)$ then we denote by  
$$A_I=\{v\in T^1X\vert_\Gamma\ \vert\ 0<\tau(v)<\infty\text{ and }I_v\text{ homotopic to }I\}$$
be the connected component of $\tau^{-1}(0,\infty)$ associated to $I$.

\begin{lem}\label{lem cumulative distribution function of arc}
    Let $I$ be a $k$-step orthogeodesic arc in $(X,\Gamma)$. 
    The set $A_I$ is open and satisfies $\lambda(\D A_I)<\infty$. Moreover, the cumulative distribution function $\phi_\lambda^{A_I}$ depends only on the lengths $\ell_X(\Gamma)$ and $\ell_X(I)$ of $\Gamma$ and $I$.
\end{lem}
\begin{proof}
    Lift $I:[0,1]\to X$ to a map $\tilde I:[0,1]\to\widetilde X$ to the universal cover. Up to isometry we might identify $\tilde X$ with the upper half place $\BH^2$ in such a way that $\tilde I(t)=i\cdot e^{\ell_X(I)\cdot t}$. When we do that, the unit half-circle $\tilde\gamma_0$ centered at $0$ is a lift of the component $\gamma_0$ of $\Gamma$ with $I(0)\in\gamma_0$ and the half-circle $\tilde\gamma_1$, again centered at $0$ and this time of radius $e^{\ell_X(I)}$, is a lift of the component $\gamma_1$ of $\Gamma$ with $I(1)\in\gamma_1$--possibly $\gamma_0=\gamma_1$, but that plays no role. 

    Now, for every $v\in A_I$ there is a proper homotopy between $I$ and the geodesic arc $I_v$. We can lift this homotopy uniquely to a homotopy between $\tilde I$ and a lift $\tilde I_v$ of $I_v$. The vector $\frac 1{\tau(v)}\frac d{dt}\tilde I_v(0)$ is a vector $\tilde v$ based at $\tilde\gamma_0$ with the property that when we flow it for time $\tau(v)$ under the geodesic flow we get a vector $\rho^{\BH^2}_{\tilde v}(\tau(v))$ over $\tilde\gamma_1$. Conversely, if we take a vector $\tilde w$ over $\tilde\gamma_0$ such that when we flow it for time $t$ we get a vector $\rho^{\BH^2}_{\tilde w}(t)$ over $\tilde\gamma_1$, when we we project $\tilde w$ down to $X$ using the covering map $\BH^2\simeq\widetilde X\to X$ then we get a vector $w\in A_I$ with $\tau(w)=t$. 

    It follows that we have an identification between $A_I$ and the set
    $$\tilde A_{\tilde I}=\{w\in T^1\BH^2\vert_{\tilde\gamma_0}\ \vert\text{ there is }t\in\BR_{>0}\text{ with }\rho^{\BH^2}_{w}(t)\in  T^1\BH^2\vert_{\tilde\gamma_1}\}$$
    As we did earlier we can identify $T^1\BH^2\vert_{\tilde\gamma_0}$ with $\tilde\gamma_0\times\BS^1$ and when we do that we can consider the measure $\tilde\lambda$ on $T^1\BH^2\vert_{\tilde\gamma_0}$ given by
    $$\int_{T^1\BH^2\vert_{\tilde\gamma_0}}f\ d\tilde\lambda=\frac 1{4\cdot \ell_X(\Gamma)}\int_\gamma\int_{\BS^1}f(x,\theta)\cdot\vert\sin(\theta)\vert\ d\theta dx$$
    for every compactly supported function $f\in C^0(T^1\BH^2\vert_{\tilde\gamma_0})$. 

    By construction, the identification between $A_I$ and $\tilde A_{\tilde I}$ is measure preserving when we endow the domain with the measure $\lambda$ and the target with the measure $\tilde\lambda$. It follows that the cumulative distribution function $\phi_\lambda^{A_I}$ agrees with the cumulative distribution function of the random variable
    $$\psi_{\ell_X(I)}:(\tilde A_{\tilde I},\tilde\lambda)\to\BR_{\ge 0}$$
    sending $w$ to $t$ if and only if $\rho^{\BH^2}_{w}(t)\in  T^1\BH^2\vert_{\tilde\gamma_1}$. In a formula this means that 
    $$\phi_\lambda^{A_I}(T)=\tilde\lambda(\psi_{\ell_X(I)}^{-1}(0,T]).$$
    Here the right side depends only on $\ell_X(\Gamma)$ and $\ell_X(I)$, as we needed to argue.
\end{proof}

We are now ready to prove Proposition \ref{prop hear orthospectrum}.

\begin{proof}[Proof of Proposition \ref{prop hear orthospectrum}]
    We will be hearing out of the melody ${\bf m}_{(X,\Gamma)}(v)=(\kappa_i,t_i)_i$ the $k$-step orthospectrum $\CG_k(X,\Gamma)$ of $(X,\Gamma)$ in an ordered way, first the length of the shortest arc, then the length of the second shortest, and so on. However, before moving on note that, since we know $\chi(X)$, we get from Lemma \ref{lem hear length} that we can hear from ${\bf m}_{(X,\Gamma)}(v)$ the total length $\ell_X(\Gamma)$.

    The key observation is that, while we cannot directly hear the measures $\nu_n$, we can hear their push forward under the random variable $\tau$ from \eqref{eq random variable tau}:
    $$\tau_*(\nu_n)=\frac 1n\sum_{i=1}^n\delta_{t_{i+k}-t_{i}}.$$
    It follows thus from Lemma \ref{lem cumulative distribution functions converge} that we can recover from ${\bf m}_{(X,\Gamma)}(v)$ the cumulative distribution function 
    $$\phi_\lambda(T)=\lim_{n\to\infty}\phi_{\nu_n}(T)=\lim_{n\to\infty}\frac 1n\vert\{i\le n\text{ with }t_{i+k}-t_{i}\le T\}.$$
    For small values of $T$ we have $\phi_\lambda(T)=0$. Indeed, density of the orbit of $v$ implies that
    $$T_1=\inf\{T\ \vert\ \phi_\lambda(T)>0\}$$
    is the length of the shortest possible essential $k$-step arc in $(X,\Gamma)$. We have found the smallest element in the orthospectrum: $T_1\in\CG_k(X,\Gamma)$.

    Suppose now by induction that we have found $T_1\le\dots\le T_r\in\CG_k(X\setminus\Gamma)$ such that there are distinct orthogeodesic oriented arcs $I_1,\dots,I_r$ in $(X,\Gamma)$ with $\ell_X(I_i)=T_i$ and that $T\ge T_r$ for all $T\in\CG_k(X\setminus\Gamma)\setminus\{T_1,\dots,T_r\}$. We do not know what are the arcs $I_1,\dots,I_r$, but this does not matter because by Lemma \ref{lem cumulative distribution function of arc} the cumulative distribution function $\phi_\lambda^{A(I_i)}$ only depends on the lengths $\ell_X(\Gamma)$ and $T_i=\ell_X(I_i)$. It follows that the cumulative distribution function of $\tau$ restricted to 
    $$U_r=T^1X\vert_\Gamma\setminus(\cup_{i=1,\dots,r}A(I_i))$$
    is given by
    $$\phi_\lambda^{U_r}=\phi_\lambda-\sum_i\phi_\lambda^{A(I_i)}$$
    and, since we can hear $\phi_\lambda$ and the $\phi_\lambda^{A(I_i)}$'s (because they only depend on $\ell_X(\Gamma)$ which as we point out earlier we can hear, and on $T_i$, which by induction we can hear), we can also hear $\phi_\lambda^{U_r}$. 

    As above, for small $T>0$, we have 
    $$\phi_\lambda^{U_r}(T)=0,$$
    and the first unaccounted for arc has length
    $$T_{r+1}=\inf\{T\ \vert\ \phi_\lambda^{U_r}(T)>0\}.$$
    We just have heard the $r+1$-th element in the $k$-step orthospectrum $\CG_k(X\setminus\Gamma)$. This concludes the proof of Proposition \ref{prop hear orthospectrum}. 
\end{proof}

Before moving on, let us add a few observations that will come in handy later on. 

\subsection*{Single note marimbas}
We suppose that $(X,\Gamma)$ is a {\em single note hyperbolic marimba}, meaning that $\Gamma$ consists just of a single curve. In fact, we will focus on the case that $(X,\Gamma)$ is a {\em separating} single note hyperbolic marimba, meaning that the curve $\Gamma$ disconnects $X$. 

\begin{prop}\label{prop single note}
    Let ${\bf m}\in(\Gamma\times\BR)^\Gamma$ such that there are a separating single note hyperbolic marimba $(X,\Gamma)$ and a random vector $v\in\CR(X)$ with ${\bf m}={\bf m}_{(X,\Gamma)}(v)$. If we denote the two components of $X\setminus\Gamma$ by $Y$ and $Y'$, and we know $\chi(X)$, then we can hear out of ${\bf m}$ the following data:
    \begin{itemize}
        \item The area and orthospectrum of $Y$.
        \item The area and orthospectrum of $Y'$.
    \end{itemize}
\end{prop}

\begin{proof}
    Note that, since $\Gamma$ consists of a single component, the melody 
    $${\bf m}={\bf m}_{(X,\Gamma)}(v)=(\kappa_i^\Gamma(v),t_i^\Gamma(v))_{i\in\BZ}$$ 
    is uniquely determined by the sequence $(t_i)_i=(t_i^\Gamma(v))_{i\in\BZ}$. Now, up to renaming the components of $X\setminus\Gamma$, we might assume that $v\in T^1Y$. This implies that
    $$\rho_v^X[t_{2i},t_{2i+1}]\subset T^1Y'\text{ and }\rho_v^X[t_{2i+1},t_{2i+2}]\subset T^1Y$$
    for all $i\ge 0$. It thus follows from the ergodicity of the geodesic flow that
    \begin{align*}
    \area(Y)
    &=\frac 1{2\pi}\area(X)\cdot\int_{T^1X}\chi_{T^1Y}\ d\Liouville\\
    &=\frac 1{2\pi}\area(X)\cdot\lim_{k\to\infty}\frac 1{t_{2k+2}}\int_0^{2t_k+2}\chi_{T^1Y}\big(\rho_v(t)\big)\ dt\\
    &=\frac 1{2\pi}\area(X)\cdot\lim_{k\to\infty}\frac 1{t_{2k+2}}\sum_{i=1}^k(t_{2i+2}-t_{2i+1}),
    \end{align*}
    where $\chi_{T^1Y}$ is the characteristic function of $T^1Y\subset T^1X$. Since we are assuming we know  $\chi(X)$, we know $\area(X)$. We have proved that ${\bf m}_{(X,\Gamma)}(v)$ determines $\area(Y)$ and hence also $\area(Y')=\area(X)-\area(Y)$.

    Now let's see how ${\bf m}_{(X,\Gamma)}(v)$ determines the orthospectrum $\CG(Y)$ of $Y$. Lacking better notation, let $T^1X\vert_\Gamma^Y$ and $T^1X\vert_\Gamma^{Y'}$ be the sets of unit tangent vectors based at $\Gamma$ which point into $Y$ and $Y'$ respectively. Both sets $T^1X\vert_\Gamma^Y$ and $T^1X\vert_\Gamma^{Y'}$ are open and their union has full $\lambda$ measure. We thus get from Lemma \ref{lem cumulative distribution functions converge} that 
    $$\lim_{n\to\infty}\phi^{T^1X\vert_\Gamma^Y}_n(t)=\phi^{T^1X\vert_\Gamma^Y}_\lambda(t)$$
    uniformly on compacta in $\BR_{\ge 0}$ where $\phi^{T^1X\vert_\Gamma^Y}_n$ and $\phi^{T^1X\vert_\Gamma^Y}_\lambda$ are the cumulative distribution functions of the random variables $\tau:(T^1X\vert_\Gamma^Y,\nu_n\vert_{T^1X\vert_\Gamma^Y})\to\BR$ and $\tau:(T^1X\vert_\Gamma^Y,\lambda\vert_{T^1X\vert_\Gamma^Y})\to\BR$, where 
    \begin{equation}
    \begin{split}
    \tau&:T^1X\vert_\Gamma\to\BR_{\ge 0},\\
    \tau&:v\mapsto\inf\left\{t>0\ \middle\vert\ \rho_v^X(t)\in T^1X\vert_\Gamma\right\}.
    \end{split}
    \end{equation}
    Noting that 
    $$\phi^{T^1X\vert_\Gamma^Y}_{2n}(T)=\frac 1{2n}\vert\{i\le n\text{ with }t_{2i}-t_{2i-1}\le T\},$$
    we get that ${\bf m}_{(X,\Gamma)}(v)$ determines the cumulative distribution function $\phi^{T^1X\vert_\Gamma^Y}_\lambda$. We can now run the exact same argument as in the proof of Proposition \ref{prop hear orthospectrum}, getting that ${\bf m}_{(X,\Gamma)}(v)$, together with the assumption that we know $\chi(X)$, determines the orthospectrum of $Y$. The same argument applies to $Y'$.
\end{proof}

\section{Proof of Theorem \ref{sat finitely many cousins}}\label{sec finitely many cousins}
In this section we prove Theorem \ref{sat finitely many cousins}, that we restate here for the convenience of the reader:

\begin{named}{Theorem \ref{sat finitely many cousins}}
    For any orientable hyperbolic marimba $(X_0,\Gamma_0)$ there are finitely many hyperbolic marimbas $(X_1,\Gamma_1),\dots,(X_k,\Gamma_k)$ such that any other orientable hyperbolic marimba $(X',\Gamma')$ with $\chi(X')=\chi(X)$ which is isomelodic to $(X_0,\Gamma_0)$ is isometric to one of the $(X_i,\Gamma_i)$'s.  
\end{named}

Let $Y_0$ be the metric completion of $X_0\setminus\Gamma_0$. The reason why we stated as Corollary \ref{kor hear orthospectrum} the fact that the melody ${\bf m}_{(X_0,\Gamma_0)}$ of a random vector $v\in\CR(X_0)$ determines the orthospectrum $\CG(Y_0)$ is that recently Le Quellec has proved in \cite[Theorem 3.1]{Nolwenn} that {\em the orthospectrum of a compact orientable hyperbolic surface $Y$ with geodesic boundary determines the metric on $Y$ up to finite indeterminacy}. Hence we get:

\begin{kor}\label{kor nolwenn}
    For any orientable hyperbolic marimba $(X_0,\Gamma_0)$ there are finitely many hyperbolic surfaces with boundary $Y_1,\dots,Y_r$ such that for any other orientable hyperbolic marimba $(X',\Gamma')$ with $\chi(X')=\chi(X)$ which is isomelodic to $(X_0,\Gamma_0)$ we have that $X'\setminus\Gamma'$ is isometric to one of the $Y_i$'s.\qed
\end{kor}

In light of Corollary \ref{kor nolwenn}, what is left to do is to prove that for any $Y\in\{Y_1,\dots,Y_r\}$ there are finitely many ways to isometrically identify boundary components, reversing orientation, to get a hyperbolic marimba $(X',\Gamma')$ which is isomelodic of $(X_0,\Gamma_0)$. 

Suppose that the boundary components of $Y$ can be isometrically paired (in an orientation reversing way) to get a hyperbolic marimba which is isomelodic of $(X_0,\Gamma_0)$. Any time we do so, we are associating a component of $\Gamma=(\gamma_1,\dots,\gamma_k)$ to every two curves which are paired, and vice versa. Since there are only finitely many ways to do this, we might assume that the boundary components of $Y$ have been labeled,
$$\D Y=\D_{\gamma_1}^-Y\cup\D_{\gamma_1}^+Y\cup\dots\cup\D_{\gamma_k}^-Y\cup\D_{\gamma_k}^+Y,$$
that the isometric gluings are such that $\D_{\gamma_i}^-Y$ is identified with $\D_{\gamma_i}^+Y$, and that the labeling has been chosen in such a way that the curve obtained by gluing those two is labeled as $\gamma_i\in\Gamma_0$. Now for any collection $\sigma=(\sigma_1,\dots,\sigma_k)$ of orientation reversing isometries
$$\sigma_i:\D_{\gamma_i}^-Y\to\D_{\gamma_i}^+Y$$
we obtain the hyperbolic marimba $(X_\sigma,\Gamma_\sigma)$ where $X_\sigma$ is the surface obtained from $Y$ by identifying $x\in\D_{\gamma_i}^-Y$ with $\sigma_i(x)\in\D_{\gamma_i}^+Y$ for $i=1,\dots,k$, and where $\Gamma_\sigma$ is the image of $\D Y$ under the map $Y\to X_\sigma$, labeled via the bijection 
$$\Gamma_0\sim\{\gamma_1^+,\dots,\gamma_k^+\}\sim\Gamma_\sigma.$$

To conclude the proof of Theorem \ref{sat finitely many cousins}, it suffices to prove that for any fixed $Y$ there are finite many choices for $\sigma=(\sigma_1,\dots,\sigma_k)$ such that the marimbas $(X_0,\Gamma_0)$ and $(X_\sigma,\Gamma_\sigma)$ are isomelodic. Note that to do that it suffices to show that there are finite sets of points 
$$F_1^-\subset\D_{\gamma_1}^-Y,F_1^+\subset\D_{\gamma_1}^+Y,\dots,F_k^-\subset\D_{\gamma_k}^-Y,F_k^+\subset\D_{\gamma_k}^+Y$$
such that if $(X_0,\Gamma_0)$ and $(X_\sigma,\Gamma_\sigma)$ are isomelodic, then 
$$\sigma_i(F_i^-)\cap F_i^+\neq\emptyset\text{ for all }i=1,\dots,k.$$

Since the argument is the same for all $i$, let us focus for the sake of concreteness on $i=1$, setting $\gamma=\gamma_i$. Let us start by picking $\beta,\delta\in\Gamma_0$ such that the three, non necessarily different, notes $\beta,\gamma,\delta$ are played in that order in the melody ${\bf m}_{(X_0,\Gamma_0)}(v)=(\kappa_i^{\Gamma_0},t_i^{\Gamma_0})_i$ of random vectors in $X$. More specifically there is some $i$ with
\begin{equation}\label{eq chord at time i}
    \kappa_i^{\Gamma_0}=\beta,\ \kappa_{i+1}^{\Gamma_0}=\gamma,\text{ and }\kappa_{i+2}^{\Gamma_0}=\delta.
\end{equation}
Whenever $i$ is such that \eqref{eq chord at time i} holds, we will say that {\em the chord $(\beta,\gamma,\delta)$ is being played at the $i$-th place}, or simply {\em at $i$}.

Consider now the quantity
$$L=\inf\{t_{i+2}^{\Gamma_0}-t_{i}^{\Gamma_0},\text{ the chord }(\beta,\gamma,\delta)\text{ is being played at }i\}.$$
Since the orbit of $v$ under the geodesic flow is dense in $T^1X$, we get that $L$ is the length of a shortest geodesic segment which 
\begin{itemize}
    \item starts at $\beta\in\Gamma_0$,
    \item ends in $\delta\in\Gamma_0$, and
    \item whose interior only meets $\Gamma_0$ at $\gamma$. 
\end{itemize}  
Now, there can be several such geodesic segments $I_1,\dots,I_r$, each one of them consisting of two geodesic segments $I_j=J_i^-\cup J_j^+$ in $X_0\setminus\Gamma_0$, the first going from $\beta$ to $\gamma$ and the second from $\gamma$ to $\delta$. For economy of notation, set $I=I_1$ and $J^\pm=J_1^\pm$.

Now, since we are assuming that $(X_\sigma,\Gamma_\sigma)$ is isomelodic of $(X_0,\Gamma_0)$, we get that in $(X_\sigma,\Gamma_\sigma)$ there are some, at least one, geodesic arcs $\hat I_1(\sigma),\dots,\hat I_s(\sigma)$ with the following properties:
\begin{itemize}
    \item $\hat I_n(\sigma)$ starts at $\beta\in\Gamma'$, ends in $\delta\in\Gamma'$, and its interior only meets $\Gamma'$ at $\gamma$.
    \item $\hat I_n(\sigma)=\hat J_n^-(\sigma)\cup\hat J_n^+(\sigma)$ where $\hat J_n^\pm(\sigma)$ are geodesic segments in $X_\sigma\setminus\Gamma_\sigma$, the first going from $\beta$ to $\gamma$ and the second from $\gamma$ to $\delta$.
    \item $\ell_{X_\sigma}(\hat J_n^\pm(\sigma))=\ell_X(J^\pm)$.
\end{itemize}
By construction, we can identify $X_\sigma\setminus\Gamma_\sigma$ and $Y$, and hence we can think of the segments $\hat J_n^\pm(\sigma)$ in there. When we do that, they have the following properties:
\begin{enumerate}
    \item $\hat J_n^-(\sigma)$ starts at $\D_\beta^- Y\cup\D_\beta^+Y$ and ends in $\D_\gamma^- Y\cup\D_\gamma^+Y$. 
    \item $\hat J_n^+(\sigma)$ starts at $\D_\gamma^- Y\cup\D_\gamma^+Y$ and ends in $\D_\delta^- Y\cup\D_\delta^+Y$.
    \item $d\sigma_\gamma$ (or $d\sigma_\gamma^{-1}$) maps the terminal vector of $\hat J_n^-(\sigma)$ to the initial vector of $J_n^+(\sigma)$.
    \item $\ell_{Y}(\hat J_n^\pm(\sigma))=\ell_{X_\sigma}(\hat J_n^\pm(\sigma))=\ell_X(J^\pm)$.
\end{enumerate}
The third property is what ensures that $\hat I_n(\sigma)=\hat J_n^-(\sigma)\cup\hat J_n^+(\sigma)$ is a smooth geodesic segment in $X_\sigma$. 

Now, the segments $\hat J_n^\pm(\sigma)$ are properly homotopic in $Y$ to geodesic segments $\tilde J_n^\pm(\sigma)$ orthogonal to the boundary. Noting that
$$\ell_{Y}(\tilde J_n^\pm(\sigma))\le \ell_{Y}(\hat J_n^\pm(\sigma))=\ell_{X_0}(J^\pm)$$
we get that, over all choices of $\sigma$ and $n$, the segments $\tilde J_n^\pm(\sigma)$ belong to a finite set of segments in $Y$.

It follows that to prove that there are finitely many choices for $\sigma_\gamma$ it suffices to prove the following claim:

\begin{named}{Claim}
    For any two geodesic segments $\tilde J^-,\tilde J^+$ in $Y$, orthogonal to $\D Y$, and with 
\begin{enumerate}
    \item $\tilde J^-$ starting at $\D_\beta^- Y\cup\D_\beta^+Y$ and ending in $\D_\gamma^- Y$, and 
    \item $\tilde J^+$ starting at $\D_\gamma^+Y$ and ending in $\D_\delta^- Y\cup\D_\delta^+Y$,
\end{enumerate}
    there are finitely many isometries $\sigma_\gamma:\D_\gamma^- Y\to \D_\gamma^+ Y$ such that there are geodesic segments $\hat J^-,\hat J^+$ in $Y$, with
    \begin{itemize}
        \item $\hat J^\pm$ is properly homotopic to $\tilde J^\pm$, 
        \item $\hat J^-$ is orthogonal to the boundary at its initial point and $\hat J^+$ at its final point,
        \item $\ell_Y(\hat J^\pm)=\ell_X(J^+)$, and
        \item $\sigma_\gamma$ maps the terminal point of $\hat J_n^-$ to the initial point of $\hat J_n^+$
    \end{itemize}
\end{named}
\begin{proof}
    Since $\hat J^{\pm}(\sigma_\gamma)$ is orthogonal to the boundary at its initial point, is proper homotopic to $\tilde J^{\pm}$, and has length $\ell_Y(\hat J^\pm(\sigma_\gamma))=\ell_X(J^+)$, we have at most two possible choices for $\hat J^-(\sigma_\gamma)$ and another two for $\hat J^+(\sigma_\gamma)$. Let $\CJ^-\subset\D_\gamma^-Y$ (resp. $\CJ^+\subset\D_\gamma^+Y$) be the endpoints of the (at most two) possible choices for $\hat J^-(\sigma_\gamma)$ (res. $\hat J^+(\sigma_\gamma)$). The map $\sigma_\gamma$ has to map some point in $\CJ^-$ to some point in $\CJ^+$. Since $\vert\CJ^-\vert,\vert\CJ^+\vert\le 2$, and since there is only an orientation reversing isometry of $\BS^1$ fixing any given point, we get that there are at most $4$ choices for $\sigma_\gamma$. 
\end{proof}

Having proved the claim, we have also proved Theorem \ref{sat finitely many cousins}.\qed

\section{Generic uniqueness}
In this section we prove Theorem \ref{sat generically unique}, or rather we reduce the proof to that of a proposition we will discuss in the next section.

\begin{named}{Theorem \ref{sat generically unique}}
    If the underlying hyperbolic surface of an orientable hyperbolic marimba $(X_0,\Gamma_0)$ is generic, then any other orientable marimba $(X',\Gamma')$ with $\chi(X')=\chi(X_0)$ which is isomelodic to $(X_0,\Gamma_0)$  is isometric to $(X_0,\Gamma_0)$.
\end{named}

As we already mentioned in the introduction, here "generic" means that $X_0$ belongs to a full measure subset of moduli space, or equivalently Teichm\"uller space. 
\medskip

Starting with the proof of Theorem \ref{sat generically unique}, we consider first that case that $(X_0,\Gamma_0)$ is a single note marimba, meaning that that $\Gamma$ consists of a single component. Our first goal is to recognize if $\Gamma$ is separating or not. 

\begin{lem}\label{lem recognize separating}
    We can recognize from the melody ${\bf m}_{(X,\Gamma)}(v)$ of a random vector in a generic single note marimba $(X,\Gamma)$ whether the curve $\Gamma$ is separating or not. 
\end{lem}


The genericity property in Lemma \ref{lem recognize separating} is the following: 
\begin{quote}
    {\em For a generic surface $X$ the following holds: if a curve $\gamma$ separates $X$ into two components $Y$ and $Y'$, then the bottom of the orthospectrum of $Y$ and $Y'$ are different.}
\end{quote} 
Here, the bottom of the orthospectrum $\CG(Y)$ of $Y$ is just the minimum of $\CG(Y)$.

\begin{proof}
    Since $\Gamma$ consists of a single component, the melody ${\bf m}_{(X,\Gamma)}(v)=(\kappa_i^\Gamma(v),t_i^\Gamma(v))_{i\in\BZ}$ is uniquely determined by the sequence $(t_i)_i=(t_i^\Gamma(v))_{i\in\BZ}$. Note also that if $\Gamma$ separates $X$ into two components $Y$ and $Y'$ then, up possibly switching their names, we have
    $$\rho_v^X[t_{2i},t_{2i+1}]\subset T^1Y\text{ and }\rho_v^X[t_{2i+1},t_{2i+2}]\subset T^1Y'$$
    for all $i$. It follows that 
    \begin{align*}
    \inf\{t_{2i+1}-t_{2i},\ i\in\BZ\}&=\text{bottom of orthospectrum of }Y\\
    \inf\{t_{2i+2}-t_{2i+1},\ i\in\BZ\}&=\text{bottom of orthospectrum of }Y'
    \end{align*}
    Since $X$ is generic, we thus get that
    \begin{equation}\label{eq even odd separating}
    \Gamma\text{ separating}\Rightarrow\inf\{t_{2i+1}-t_{2i},\ i\in\BZ\}\neq \inf\{t_{2i+2}-t_{2i+1},\ i\in\BZ\}.
    \end{equation}

    Suppose now that $\Gamma$ is non-separating and consider the cover $\hat X$ of $X$ corresponding to the kernel of the homomorphism
    $$\pi_1(X)\to\BZ/2\BZ,\ \alpha\mapsto\iota(\alpha,\Gamma)\mod(2).$$
    The preimage $\hat\Gamma$ of $\gamma$ under $\pi:\hat X\to X$ is a multicurve consisting of two curves which separate $\hat X$ into two isometric copies $\hat Y,\hat Y'$ of $X\setminus\Gamma$. Associating to $\hat\Gamma$ two musical notes of your choosing--let's go for Fa and La--consider the hyperbolic marimba $(\hat X,\hat\Gamma)$. Let $\hat v\in T^1\hat X$ be a vector with $d\pi(\hat v)=v$ and consider its melody 
    $${\bf m}_{(\hat X,\hat\Gamma)}(\hat v)=(\kappa_i^{\hat\Gamma}(\hat v),t_i^{\hat\Gamma}(\hat v))_i$$
    The basic observations are that the sequences of times of $v$ and $\hat v$ are identical, that is
    \begin{equation}\label{eq identical times}
        (t_i^{\hat\Gamma}(\hat v))_i=(t_i)_i,
    \end{equation}
    and that up possibly switching the names of $\hat Y$ and $\hat Y'$, we have
    $$\rho_v^X[t_{2i}^\Gamma(\hat v),t^\Gamma_{2i+1}(\hat v)]\subset T^1\hat Y\text{ and }\rho_v^X[t^\Gamma_{2i+1}(\hat v),t^\Gamma_{2i+2}(\hat v)]\subset T^1\hat Y'$$
    for all $i$. This last property means that, as above, we have 
    \begin{align*}
    \inf\{t^\Gamma_{2i+1}(\hat v)-t^\Gamma_{2i}(\hat v),\ i\in\BZ\}&=\text{bottom of orthospectrum of }\hat Y\\
    \inf\{t^\Gamma_{2i+2}(\hat v)-t^\Gamma_{2i+1}(\hat v),\ i\in\BZ\}&=\text{bottom of orthospectrum of }\hat Y'
    \end{align*}
    Now, we have however that $\hat Y$ and $\hat Y'$ are isometric, meaning that 
    $$\inf\{t^\Gamma_{2i+1}(\hat v)-t^\Gamma_{2i}(\hat v),\ i\in\BZ\}=\inf\{t^\Gamma_{2i+2}(\hat v)-t^\Gamma_{2i+1}(\hat v),\ i\in\BZ\}.$$
    Combining this with \eqref{eq identical times}, we get
    \begin{equation}\label{eq even odd non-separating}
    \Gamma\text{ non-separating}\Rightarrow\inf\{t^\Gamma_{2i+1}-t^\Gamma_{2i},\ i\in\BZ\}= \inf\{t^\Gamma_{2i+2}-t^\Gamma_{2i+1},\ i\in\BZ\}.
    \end{equation}
    The claim of the lemma follows directly from \eqref{eq even odd separating} and \eqref{eq even odd non-separating}.
\end{proof}

Once we know that the melody ${\bf m}_{(X,\Gamma)}(v)$ of a random vector in a generic single note marimba $(X,\Gamma)$ determines whether $\Gamma$ separates or not, let us prove that ${\bf m}_{(X,\Gamma)}(v)$ actually determines $X\setminus\Gamma$ up to isometry.

\begin{lem}\label{lem recognize complement up to isometry}
    The melody ${\bf m}_{(X,\Gamma)}(v)$ of random vector in a generic orientable single note hyperbolic marimba $(X,\Gamma)$ determines $X\setminus\Gamma$ up to isometry.
\end{lem}


The genericity property in Lemma \ref{lem recognize complement up to isometry} is the following: 
\begin{quote}
    ($\star$)    {\em for a generic orientable surface $X$ the following holds: if $\Gamma\subset X$ is a simple closed geodesic then the topology of $Y=X\setminus\Gamma$ together with the orthospectrum $\CG(Y')$ of every connected component $Y'$ of $Y$, determine $Y$ up to isometry.}
\end{quote}
This is basically a consequence of yet another result of Le Quellec, namely Theorem 4.1 in \cite{Nolwenn}: {\em Generic surfaces in the Teichm\"uller space $\CT_{g,b}$ of connected, orientable genus $g$ surfaces with $b$ boundary components, are determined by their orthospectrum up to isometry.} However, as stated we cannot apply directly Le Quellec's theorem when $\Gamma$ is non-separating because $\CT_b(X\setminus\Gamma)$ is not a generic subset of $\CT(X\setminus\Gamma)$. However, the statement of her theorem still holds when we restrict to that subset. In the appendix we will explain how to adapt Le Quellec's argument. For now, let's us use ($\star$) to prove Lemma \ref{lem recognize complement up to isometry}:

\begin{proof}
    By Lemma \ref{lem recognize separating} we can recognize whether $\Gamma$ is separating or not. 
    Suppose first that $\Gamma$ is non-separating. We get from Corollary \ref{kor hear orthospectrum} that ${\bf m}_{(X,\Gamma)}(v)$ determines the orthospectrum of $X\setminus\Gamma$ and hence from ($\star$) that $X\setminus\Gamma$ is determined by ${\bf m}_{(X,\Gamma)}(v)$.

    If $\Gamma$ separates, let $Y_1,Y_2$ be the two components of $X\setminus\Gamma$. Up to switching the names we might assume that $v\in T^1Y_1$. We get thus from Proposition \ref{prop single note} that ${\bf m}_{(X,\Gamma)}(v)$ determines the areas $\area(Y_1)$ and $\area(Y_2)$ and orthospectra $\CG(Y_1)$ and $\CG(Y_2)$ of $Y_1$ and $Y_2$. Since $Y_1$ is orientable and has a single boundary component, we get from Gau\ss -Bonnet that its area determines its topology. The same argument applies to $Y_2$. Now, knowing the topology of the components of $X\setminus\Gamma$ and their orthospectra, we get from ($\star$) that ${\bf m}_{(X,\Gamma)}(v)$ determines $X\setminus\Gamma$ up to isometry.
\end{proof}

Still assuming that $(X_0,\Gamma_0)$ is a single note hyperbolic marimba, we know at this point that if $X\in\CT(X_0)$ is generic and $(X',\Gamma')$ is isomelodic of $(X_0,\Gamma_0)$ then $X'\setminus\Gamma'$ is isometric to 
$$Y_0=X_0\setminus\Gamma_0.$$ 
Since both $X_0$ and $X'$ are obtained from $Y_0$ by gluing back both boundary components via an orientation reversing isometry, it remains to prove that generically the gluing map is determined by the melody of a random vector in $(X_0,\Gamma_0)$. We remind the reader that $\CT_b(Y_0)$ is the Teichm\"uller space of hyperbolic metrics on $Y_0$ whose boundary consists of two geodesics of the same length and that, by Proposition \ref{prop FN}, there is an analytic isomorphism
$$\Phi:\CT_b(Y_0)\times\BR^{\Gamma_0}\to\CT(X_0,\Gamma_0).$$
In the next section we will prove the following:
    
\begin{prop}\label{prop wolpert}
    Let $(X_0,\Gamma_0)$ be an orientable single note marimba, set $Y_0=X_0\setminus\Gamma_0$, and let $Y\in\CT_b(Y_0)$ be generic. For all $\theta\in\BR$ the following holds: if $\theta'\in\BR$ is such that the marimbas $\Phi(Y,\theta)$ and $\Phi(Y,\theta')$ are isomelodic, then 
    \begin{itemize}
        \item $\theta'=\theta\mod(\ell_{\Phi(Y,\theta)}(\Gamma))$ if no component of $Y_0$ is a one-holed torus.
        \item $\theta'=\theta\mod(\frac 12\ell_{\Phi(Y,\theta)}(\Gamma))$ if some component of $Y_0$ is a one-holed torus.
    \end{itemize}
    Here $\Phi$ is as in Proposition \ref{prop FN}.
\end{prop}

Assuming for the time being Proposition \ref{prop wolpert}, we conclude the proof of Theorem \ref{sat generically unique}:

\begin{proof}[Proof of Theorem \ref{sat generically unique}]
Let's recap the situation if $(X_0,\Gamma_0)$ is a single note marimba. Supposing $X_0$ generic, we get from Lemma \ref{lem recognize complement up to isometry} that $Y_0=X_0\setminus\Gamma_0$ is isometric to $X'\setminus\Gamma'$ for any isomelodic hyperbolic marimba $(X',\Gamma')$. It follows that when we identify
$$\Phi:\CT_b(Y_0)\times\BR\to\CT(X_0,\Gamma_0),\ (Y,\theta)\mapsto \Phi((Y,\theta))$$
as in Proposition \ref{prop FN},  we get that there are $\theta,\theta'\in\BR$ and isometries
$$(X_0,\Gamma_0)\simeq\Phi(Y_0,\theta)\text{ and }(X',\Gamma')\simeq\Phi(Y_0,\theta').$$
Now, since $X_0$ is generic, $Y_0$ is generic in $\CT_b(Y_0)$. We thus get from Proposition \ref{prop wolpert} that
    \begin{itemize}
        \item $\theta$ and $\theta'$ agree modulo $\ell_{\Phi(Y_0,\theta)}(\Gamma)=\ell_{X_0}(\Gamma)$ if no component of $Y_0$ is a one-holed torus, and that 
        \item $\theta$ and $\theta'$ agree modulo $\frac 12\ell_{\Phi(Y_0,\theta)}(\Gamma)=\frac 12\ell_{X_0}(\Gamma)$ if some component of $Y_0$ is a one-holed torus.
    \end{itemize}
It thus follows from the addendum to Proposition \ref{prop FN} that the marimbas $(X_0,\Gamma_0)\simeq\Phi(Y_0,\theta)$ and $(X',\Gamma')\simeq\Phi(Y_0,\theta')$ are isometric. We have proved Theorem \ref{sat generically unique} for single note marimbas.

Now to the general case. Suppose that $(X_0,\Gamma_0)$ and $(X',\Gamma')$ are isomelodic, let $\gamma\in\Gamma_0$ be a component, and let $\gamma'\in\Gamma'$ be the corresponding component, that is the one which plays the same note. Noting the we recover the melodies for $(X_0,\gamma)$ from those of $(X_0,\Gamma_0)$ by simply muting out all other notes, we get that the one-note marimbas $(X_0,\gamma)$ and $(X',\gamma')$ are again isomelodic. Since we have already proved the theorem for one-note marimbas, we get that there is an isometry
$$\phi_\gamma:(X_0,\gamma)\to(X',\gamma').$$
Proceeding in this way for all notes we get a collection of isometries $\{\phi_\gamma:X_0\to X',\ \gamma\in\Gamma_0\}$ such that each $\phi_\gamma$ sends $\gamma\in\Gamma$ to the corresponding component of $\Gamma'$.

If $X_0$ has at least genus $3$ and is generic, then the group $\Isom(X_0)$ of all isometries $X_0\to X_0$ is reduced to the identity. It follows that $\phi_\gamma=\phi_\eta$ for any two components $\gamma,\eta$ of $\Gamma$. In other words, the isometry $\phi_\gamma:X_0\to X'$ sends $\Gamma$ to $\Gamma'$ in a label preserving way.

In genus $2$, we still claim that $\phi_\gamma$ sends $\Gamma_0$ to $\Gamma'$ in a label preserving way. Indeed, if $\eta\in\Gamma_0$ is a second component then consider again the isometry $\phi_\gamma^{-1}\circ\phi_\eta:X_0\to X_0$. A generic surface $X_0$ of genus $2$ has exactly two isometries: the identity and the hyperelliptic involution $\iota_{X_0}:X_0\to X_0$. If $\phi_\eta^{-1}\circ\phi_\gamma=\Id$ then we have nothing to say. When $\phi_\eta^{-1}\circ\phi_\gamma=\iota_{X_0}$ we still have that 
$$\phi_\gamma(\eta)=\phi_\eta(\iota_X(\eta))=\phi_\eta(\eta)$$
because the hyperelliptic involution preserves every simple closed geodesic in $X_0$. Having proved that that $\phi_\gamma$ sends $\Gamma$ to $\Gamma'$ in a label preserving way, we are done with the proof of Theorem \ref{sat generically unique}.
\end{proof}

It remains to prove Proposition \ref{prop wolpert}. This will be the content of the next and final section of this paper.

\section{A variation on an argument of Wolpert's}

Although with differences, the proof of Proposition \ref{prop wolpert} will follow the basic strategy of Wolpert's proof of the fact that generic surfaces are determined by the their unmarked length spectra \cite{Wolpert} (see also \cite{Buser}). This is also the same kind of argument used by Le Quellec \cite{Nolwenn} to prove that generic surfaces with boundary are determined by their orthospectrum. 
\medskip

Fix from now on a single-note hyperbolic marimba $(X_0,\Gamma_0)$ and, as always let $Y_0$ be the metric completion of $X_0\setminus\Gamma_0$. To relax notation we will denote images under the map 
$$\Phi:\CT_b(Y_0)\times\BR\to\CT(X_0,\Gamma_0)$$
from Proposition \ref{prop FN} simply by $(Y,\theta)$ instead of $\Phi(Y,\theta)$.

From now on denote by $\CA_2=\CA_2(X_0,\Gamma_0)$ be the set of 2-step arcs in $(X_0,\Gamma_0)$.

\begin{lem}\label{lem worlpert analytic}
    Let $\iota:\CA_2\to\CA_2$ be a map. The set 
    $$\CV(\iota)=\left\{(Y,\theta)\in \CT_b(Y_0)\times\BR\ \middle\vert\begin{array}{l}\text{there is }\theta'\in\BR\text{ such that}\\
    \ell_{(Y,\theta')}(\iota(\tau))=\ell_{(Y,\theta)}(\tau)\text{ for all }\tau\in\CA_2\end{array}\right\}$$
    is the vanishing locus of an analytic map $\CT_b(Y)\times\BR\to\BR$. We moreover have that 
    $$\CV(\iota)=\left\{(Y,\theta)\in \CT_b(Y_0)\times\BR\ \middle\vert\begin{array}{l}\text{there is }\theta'\in\BR\text{ such that}\\
    \ell_{(Y,\theta')}(\iota(\tau))=\ell_{(Y,\theta)}(\tau)\text{ for all }\tau\in\CA'\end{array}\right\}$$
    for some finite subset $\CA'\subset\CA_2$.
\end{lem}
\begin{proof}
    Number the elements of $\CA_2$ as $\tau_0,\tau_1,\tau_2,\dots$ Since $\CT_b(Y)\times\BR$ is connected, the claim follows from the Noetherian property of the ring of analytic functions when we prove that for all $k\ge 1$ the set
    $$\CV_k(\iota)=\left\{(Y,\theta)\in \CT_b(Y_0)\times\BR\ \middle\vert\begin{array}{l}\text{there is }\theta'\in\BR\text{ such that}\\
    \ell_{(Y,\theta')}(\iota(\tau))=\ell_{(Y,\theta)}(\tau)\text{ for all }\tau\in\{\tau_0,\dots,\tau_k\}\end{array}\right\}$$
    is the vanishing locus of a non-zero analytic map. By Lemma \ref{lem analytic function} there is an analytic map
    $$P_k:\CT_b(Y_0)\times\BR_{>0}^{k+1}\to\BR^{k}$$
    such that 
    $$\left\{(Y,\cosh(\ell_{(Y,\theta')}(\iota(\tau_0))),\dots,\cosh(\ell_{(Y,\theta')}(\iota(\tau_k))))\in\CT_b(Y_0)\times\BR_{>0}^{k+1}\ \middle\vert\  \theta'\in\BR
    \right\}=\left\{P_k=0\right\}.$$
    Consider also the analytic map 
    \begin{align*}
        \phi&:\CT_b(Y_0)\times\BR\to\CT_b(Y_0)\times\BR_{>0}^{k+1}\\
        \phi&:(Y,\theta)\mapsto(Y,\cosh(\ell_{(Y,\theta)}(\tau_0)),\dots,\cosh(\ell_{(Y,\theta)}(\tau_r))),
    \end{align*}
    and note that $\CV_k(\iota)$ is nothing other than the vanishing locus of the analytic map $P_k\circ\phi:\CT_b(Y_0)\times\BR\to\BR^k$.
\end{proof}

It follows from Lemma \ref{lem worlpert analytic} that, for any map $\iota:\CA_2\to\CA_2$, the set $\CV(\iota)$ is either equal to $\CT_b(Y_0)\times\BR$, or is a lower-dimensional analytic subset. Our next goal is to prove that any $\iota:\CA_2\to\CA_2$ with $\CV(\iota)=\CT_b(Y_0)\times\BR$ preserves certain subsets $\CA_{\alpha,\beta}$ of $\CA_2$. These subsets are obtained as follows.

Let $\alpha,\beta\subset Y_0$ be two proper simple disjoint unoriented essential arcs with $\D\alpha\subset\D^-Y_0$ and $\D\beta\subset\D^+Y_0$. We will say that a 2-step arc in $(X,\Gamma)$ {\em is of type $(\alpha,\beta)$} when $\tau\setminus\Gamma$ is properly homotopic to $\alpha\cup\beta$. Denote by 
$$\CA_{\alpha,\beta}=\{\tau\in\CA_2\text{ of type }(\alpha,\beta)\}$$ 
the set of all such arcs.

\begin{lem}\label{lem iota preserves arcs}
    If $\iota:\CA_2\to\CA_2$ is such that $\CV(\iota)=\CT_b(Y_0)\times\BR$, then we have $\iota(\CA_{\alpha,\beta})\subset\CA_{\alpha,\beta}$ for any two proper simple disjoint unoriented essential arcs $\alpha,\beta\subset Y_0$ with $\D\alpha\subset\D^-Y_0$ and $\D\beta\subset\D^+Y_0$
\end{lem}
\begin{proof}
    Suppose that $\tau\in\CA_{\alpha,\beta}$. Since $\alpha,\beta$ are simple and disjoint there is some $Y\in\CT_b(Y_0)$ such that $\alpha,\beta$ are respectively the unique shortest homotopy classes of arcs in $Y$ with at least an endpoint in $\D^-Y_0$ and $\D^+Y_0$. There is also $\theta$ such that 
    $$\ell_{(Y,\theta)}(\tau)=\ell_Y(\alpha)+\ell_Y(\beta).$$
    Now, by assumption there is $\theta'$ such that 
    $$\ell_{(Y,\theta')}(\iota(\tau))=\ell_{(Y,\theta)}(\tau).$$
    Let $\hat\alpha,\hat\beta$ be the two arcs $\iota(\tau)\setminus\Gamma$ and note that, up to exchanging the names, we can assume that $\hat\alpha\cap\D^-Y_0\neq\emptyset$ and $\hat\beta\cap\D^+Y_0\neq\emptyset$. By the choice of $Y$ this means that 
    $$\ell_Y(\hat\alpha)\ge\ell_Y(\alpha)\text{ and }\ell_Y(\hat\beta)\ge\ell_Y(\beta)$$
    with equality if and only if $\hat\alpha=\alpha$ and $\hat\beta=\beta$. However, we must have equality because
    $$\ell_Y(\hat\alpha)+\ell_Y(\hat\beta)\le\ell_{(Y,\theta')}(\iota(\tau))=\ell_{(Y,\theta)}(\tau)=\ell_Y(\alpha)+\ell_Y(\beta).$$
    This proves that $\hat\alpha=\alpha$ and $\hat\beta=\beta$, as claimed. 
\end{proof}

Recall that we are writing $(Y,\theta)$ for $\Phi(Y,\theta)$, where $\Phi:\CT_b(Y_0)\times\BR^{\Gamma_0}\to\CT(X_0,\Gamma_0)$ is the map from Proposition \ref{prop FN}. In particular, the underlying surface of the hyperbolic marimba $(Y,\theta)=\Phi(Y,\theta)$ is obtained by identifying $\D^-Y$ and $\D^+Y$ via an orientation reversing isometry
$$\sigma_Y^{\theta}:\D_\gamma^-Y\to\D_\gamma^+Y.$$
Moreover, for $\theta,\theta'\in\BR$ the isometries $\sigma_Y^\theta$ and $\sigma_Y^{\theta'}$ are related by
\begin{equation}\label{eq relation between gluings}
    \sigma_{Y}^{\theta'}(x)=\sigma_{Y}^\theta(x+\theta'-\theta)
\end{equation}
where $x+\theta'-\theta$ is the point in $\D^-_\gamma Y$ that one reaches from $x$ by running for time one with signed speed $\theta'-\theta$. 

Continuing with our previous discussion, let $\alpha,\beta\subset Y_0$ be proper simple disjoint arcs with $\D\alpha\subset\D^-Y_0$ and $\D\beta\subset\D^+Y_0$. In the sequel we will be interested in the set
$$\calD_{\alpha,\beta}^Y(\theta)
=\{d_{\D^-Y}(a,b)\text{ where }a\in\D\alpha_Y\text{ and }b\in(\sigma_Y^\theta)^{-1}(\D\beta_Y)\}\}$$
of distances in $\D^-Y$ between points in $\D\alpha_Y$ and in $(\sigma_Y^\theta)^{-1}(\D\beta_Y)$. Here we are denoting, as we already did before, by $\alpha_Y,\beta_Y$ the orthogeodesic representatives in $Y\in\CT_b(Y_0)$ of $\alpha$ and $\beta$.

\begin{lem}\label{lem::differences_theta_thetaprime}
    Suppose that $\iota:\CA_2\to\CA_2$ is such that $\CV(\iota)=\CT_b(Y_0)\times\BR$ and let $Y\in\CT_b(Y_0)$ and $\theta,\theta'\in\BR$ such that $\ell_{(Y,\theta')}(\iota(\tau))=\ell_{(Y,\theta)}(\tau)$ for all $\tau\in\CA_2$. Then we have
    $$\calD_{\alpha,\beta}^Y(\theta)=\calD_{\alpha,\beta}^Y(\theta')$$
    for any two proper simple disjoint unoriented essential arcs $\alpha,\beta\subset Y_0$ with $\D\alpha\subset\D^-Y_0$ and $\D\beta\subset\D^+Y_0$.
\end{lem}
\begin{proof}
    Every $2$-step arc $\tau\in\CA_{\alpha,\beta}$ in $(Y,\theta)=\Phi(Y,\theta)$ of type $(\alpha,\beta)$ is homotopic to a uniquely determined path of the form:
    \begin{itemize}
        \item First run along $\alpha_Y$, arriving to a point $[p_\tau]\in\D\alpha_Y$,
        \item then run along $\D^-Y$ for some signed time $t_\tau(\theta)\in\BR$, arriving to a point $q_\tau\in (\sigma_Y^\theta)^{-1}(\D\beta_Y)$, 
        \item and finally run $\beta_Y$ starting at the point $\sigma_Y^\theta(q_\tau)$.
    \end{itemize}
    Note that $\vert t_\tau(\theta)\vert\mod(\ell(\D^-Y))=d_{\D^-Y}(p_\tau,p_\tau)\in\calD_{\alpha,\beta}^Y(\theta)$. Conversely, every element of $\calD_{\alpha,\beta}^Y(\theta)$ arises in this way. Lemma \ref{lem::differences_theta_thetaprime} will thus follow once we prove that
    \begin{equation}\label{eq blablabla123}
        \{t_\tau(\theta)\text{ with }\tau\in\CA_{\alpha,\beta}\}\subset \{t_\tau(\theta'),-t_\tau(\theta')\text{ with }\tau\in\CA_{\alpha,\beta}\}.
    \end{equation}
    The basic observation is that, with $a=\sinh(\ell_Y(\alpha_Y))\cdot\sinh(\ell_Y(\beta_Y))$ and $b=\cosh(\ell_X(\alpha_Y))\cdot\cosh(\ell_X(\beta_Y))$, we get from \eqref{eq thm 244 buser} that for all $\tau,\tau'\in\CA_{\alpha,\beta}$ we have
    $$\cosh(\ell_{(Y,\theta)}(\tau))=a\cdot\cosh(t_\tau(\theta))+b\text{ and }\cosh(\ell_{(Y,\theta')}(\tau'))=a\cdot\cosh(t_{\tau'}(\theta'))+b.$$
    It follows that we have
    $$\cosh(t_{\tau}(\theta))=\cosh(t_{\iota(\tau)}(\theta'))$$
    for all $\tau\in\CA_{\alpha,\beta}$, from where we get $t_{\tau}(\theta)=\pm t_{\iota(\tau)}(\theta')$. Since this holds for all $\tau\in\CA_{\alpha,\beta}$, \eqref{eq blablabla123} follows, concluding the proof of the lemma.
\end{proof}

Lemma \ref{lem::differences_theta_thetaprime} implies that there are strong relations between $\theta$ and $\theta'$. To exploit them we need to produce suitable arcs. Once again we need some terminology. First recall that the orientation of $X$ induces one on $Y$, and hence on $\D Y$. More specifically $\D Y$ is oriented so that when we promenade along it in positive direction, the interior stays to our right. This orientation induces a cyclic ordering on each component of $\D Y$. 
\medskip

We will say that a sequence $(E_n)_n$ of finite subsets of $\D Y$ is {\em positively convergent} if
\begin{enumerate}
    \item there is $x\in\D Y$ with $\lim_nx_n=x$ for any sequence $(x_n)$ with $x_n\in E_n$, and 
    \item there is $n_0$ such that for all $n>m>n_0$ any tuple $(x_m,x_n,x)$ with $x_m\in E_m$ and $x_n\in E_n$ is positive with respect to the cyclic ordering of the component of $\D Y$ containing $x$.
\end{enumerate}
If we replace (2) by the condition that the tuples $(x_m,x_n,x)$ are negative, then we say that $(E_n)_n$ is {\em negatively convergent}.
\medskip

With this notation, we have the following:

\begin{lem}\label{lem get arcs1}
    Suppose that no connected component of $Y_0$ is a one-holed torus. There are sequences $(\alpha^{n})_{n},(\beta^{n})_{n}$ of simple arcs with $\D\alpha^{n}\subset\D^-Y_0$ and $\D\beta^{n}\subset\D^+Y_0$ with the following properties:
    \begin{itemize}
        \item $\alpha^{n}$ and $\beta^{n}$ are disjoint for all $n$,
        \item both sequences $(\D\alpha^n_Y)_n$ and $(\D\beta^n_Y)_n$ are positively convergent in $\D Y$ for every $Y\in\CT_b(Y_0)$.
    \end{itemize}
\end{lem}
\begin{proof}
    Suppose that $Y_0$ is connected and take two disjoint arcs $\alpha,\beta$ with $\D\alpha\subset\D^-Y_0$ and $\D\beta\subset\D^+Y_0$. Then we can find a closed curve $\eta\subset Y_0$ which meets both of them exactly twice, and with the property that when we cut $\alpha$ (resp. $\beta$) along $\eta$, the initial and final arcs are properly homotopic in $Y_0\setminus\eta$. Denote by $\delta_\eta$ the left-twist along $\eta$.
    If $(t_n)_n$ is any sequence which tends to $\infty$ fast enough, then the arcs $\alpha^n=\delta_\eta^{t_n}(\alpha)$ and $\beta^n=\delta_\eta^{t_n}(\beta)$ have the properties we wanted for any $Y\in\CT_b(Y_0)$.

\begin{figure}[h]
\leavevmode \SetLabels
\endSetLabels
\begin{center}
\AffixLabels{\centerline{\includegraphics[width=0.4\textwidth]{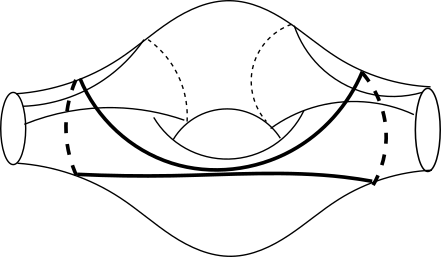}}\hspace{1cm}}
\vspace{-24pt}
\end{center}
\caption{The segment ${\bf T}\subset\eta$ used to calculate the twist is bold.}
\label{fig music twist1}
\end{figure}

    If $Y_0$ is disconnected we take $\alpha^n=\delta_{\eta_1}^{t_n}(\alpha)$ and $\beta^n=\delta^{t_n}_{\eta_2}(\beta)$ where $\delta_{\eta_1}$ and $\delta_{\eta_2}$ are now the left-twists along two curves in $Y_0$ which again have the property that when we cut $\alpha$ (resp. $\beta$) along them, the initial and final arcs are properly homotopic in $Y_0\setminus\eta_1$ (resp. $Y_0\setminus\eta_2$).
\end{proof}

When one of the connected components of $Y_0$ is a one-holed torus, then Lemma \ref{lem get arcs1} fails because the two endpoints of a simple arc are always opposite to each other. To deal with that issue we replace $\D^\pm Y$ by the quotient $(\D^\pm Y)/_{\tiny{\rm anti}}$ where we identify antipodal points. The circles $(\D^\pm Y)/_{\tiny{\rm anti}}$ are naturally oriented. We write $(\D Y)/_{\tiny{\rm anti}}=(\D^-_Y)/_{\tiny{\rm anti}}\sqcup(\D^+Y)/_{\tiny{\rm anti}}$. With this notation in place, we have the following version of Lemma \ref{lem get arcs1}, whose proof we leave to the reader:

\begin{lem}\label{lem get arcs2}
    There are sequences $(\alpha^{n})_{n},(\beta^{n})_{n}$ of simple arcs with $\D\alpha^{n}\subset\D^-Y_0$ and $\D\beta^{n}\subset\D^+Y_0$ with the following properties:
    \begin{itemize}
        \item $\alpha^{n}$ and $\beta^{n}$ are disjoint for all $n$,
        \item both sequences $((\D\alpha^n_Y)/_{\tiny{\rm anti}})_n$ and $((\D\beta^n_Y)/_{\tiny{\rm anti}})_n$ are positively convergent in $(\D Y)/_{\tiny{\rm anti}}$ for every $Y\in\CT_b(Y_0)$.\qed
    \end{itemize}
\end{lem}

We are now ready to prove Proposition \ref{prop wolpert}, which we restate here for the convenience of the reader:

\begin{named}{Proposition \ref{prop wolpert}}
    Let $(X_0,\Gamma_0)$ be a single note marimba, set $Y_0=X_0\setminus\Gamma_0$, and let $Y\in\CT_b(Y_0)$ be generic. For all $\theta\in\BR$ the following holds: if $\theta'\in\BR$ is such that the marimbas $\Phi(Y,\theta)$ and $\Phi(Y,\theta')$ are isomelodic, then 
    \begin{itemize}
        \item $\theta'=\theta\mod(\ell_{\Phi(Y,\theta)}(\Gamma))$ if no component of $Y_0$ is a one-holed torus.
        \item $\theta'=\theta\mod(\frac 12\ell_{\Phi(Y,\theta)}(\Gamma))$ if some component of $Y_0$ is a one-holed torus.
    \end{itemize}
    Here $\Phi$ is as in Proposition \ref{prop FN}.
\end{named}
\begin{proof}
    We get from Proposition \ref{prop hear orthospectrum} that the melody of a any random vector determines the $k$-step orthospectrum of the marimba. It follows that if for some $Y\in\CT_b(Y_0)$ and $\theta,\theta'\in\BR$ the marimbas $(Y,\theta)=\Phi(Y,\theta)$ and $(Y,\theta')=\Phi(Y,\theta')$ are isomelodic, then we have
    $$\CG_2(Y,\theta)=\CG_2(Y,\theta').$$
    Said differently, their is a map $\iota:\CA_2(X_0,\Gamma_0)\to\CA_2(X_0,\Gamma_0)$ with
    \begin{equation}\label{eq map iota isomelodic}
        \ell_{(Y,\theta)}(\tau)=\ell_{(Y,\theta')}(\iota(\tau))\text{ for all }\tau\in\CA_2=\CA_2(X_0,\Gamma_0).
    \end{equation}
    The map $\iota$ depends here on $Y,\theta$ and $\theta'$. 
    
    For any such map $\iota$ consider as in Lemma \ref{lem worlpert analytic} the set
        $$\CV(\iota)=\left\{(Y,\theta)\in \CT_b(Y_0)\times\BR\ \middle\vert\begin{array}{l}\text{there is }\theta'\in\BR\text{ such that}\\
    \ell_{(Y,\theta')}(\iota(\tau))=\ell_{(Y,\theta)}(\tau)\text{ for all }\tau\in\CA_2\end{array}\right\}$$
    and let 
    $$\CI=\{\iota:\CA_2(X_0,\Gamma_0)\to\CA_2(X_0,\Gamma_0)\text{ such that }\CV(\iota)\neq\CT_b(Y_0)\times\BR\}.$$ 
    Although the set $\CI$ is uncountable, it follows from Lemma \ref{lem worlpert analytic} that the set $\{\CV(\iota),\ \iota\in\CI\}$ consists of countably many proper analytic subsets of $\CT_b(Y_0)\subset\BR\simeq\BR^{3\vert\chi(X)\vert}$. It follows that $\CW=\cup_{\iota\in\CI}\CV(\iota)$ is negligible. In other words, if for a generic $Y\notin\CW$ we have $\theta,\theta'\in\BR$ such that $(Y,\theta)$ and $(Y,\theta')$ are isomelodic, then there is map $\iota:\CA_2(X_0,\Gamma_0)\to\CA_2(X_0,\Gamma_0)$ with $\CV(\iota)=\CT_b(Y_0)\times\BR$ and satisfying \eqref{eq map iota isomelodic}. We are going to show that this implies that $\theta$ and $\theta'$ are related as predicted by the claim of the proposition we are trying to prove.
    
    Since $\CV(\iota)=\CT_b(Y_0)\times\BR$, we get from Lemma \ref{lem::differences_theta_thetaprime} that 
        $$\calD_{\alpha,\beta}^Y(\theta)=\calD_{\alpha,\beta}^Y(\theta')$$
    for any two proper simple disjoint unoriented essential arcs $\alpha,\beta\subset Y_0$ with $\D\alpha\subset\D^-Y_0$ and $\D\beta\subset\D^+Y_0$. Here 
    $$\calD_{\alpha,\beta}^Y(\theta)
    =\{d_{\D^-Y}(a,b)\text{ where }a\in\D^-Y(\alpha)\text{ and }b\in(\sigma_Y^\theta)^{-1}(\D\beta_Y)\}\}$$
    is the set of distances in $\D^-Y$ between points in $\D\alpha_Y$ and in $(\sigma_Y^\theta)^{-1}(\D\beta_Y)$, where $\sigma_Y^\theta:\D^-Y\to\D^+Y$ is the orientation reversing isometric gluing we use to obtain the marimba $(Y,\theta)=\Phi(Y,\theta)$ from $Y$. 

    Let's now suppose that no component of $Y_0$ is a one-holed torus and let $(\alpha^{n})_{n}$ and $(\beta^{n})_{n}$ be sequences of arcs as provided by Lemma \ref{lem get arcs1}. Since the two arcs $\alpha^n$ and $\beta^n$ are simple and disjoint, we have  
    \begin{equation}\label{eq want to be done with this}
    \calD_{\alpha^{n},\beta^{n}}^Y(\theta)=\calD_{\alpha^{n},\beta^{n}}^Y(\theta')\text{ for all }n.
    \end{equation}
    Recalling that the two sequences $(\D\alpha^n_Y)_n$ and $(\D\beta^n_Y)_n$ are positively convergent, let
    $A\in\D^-Y$ and $B\in\D^+Y$ be such that 
    $$A=\lim_na_n \text{ and }B=\lim_nb_n$$
    for any sequences $(a_n)$ and $(b_n)$ with $a_n\in\D\alpha^n_Y$ and $b_n\in\D\beta^n_Y$. Denoting by 
    $$B^\theta=(\sigma_Y^\theta)^{-1}(B)$$
    we get from \eqref{eq want to be done with this} that
    \begin{equation}\label{eq almost done with this}
    d_{\D^-Y}(A,B^\theta)=d_{\D^-Y}(A,B^{\theta'}).
    \end{equation}
    Seeking a contradiction we suppose that $B^\theta\neq B^{\theta'}$. To begin with note hat, up to interchanging their roles, we might assume that the tuple $(B^\theta,A,B^{\theta'})$ is positively oriented with respect to the cyclic ordering induced by the orientation of $\D^-Y$. The key observation now is that while $(\D\alpha^n_Y)_n$ is positively convergent to $A$, the sequences $((\sigma_Y^\theta)^{-1}(\D\beta^n_Y))_n$ and $((\sigma_Y^{\theta'})^{-1}(\D\beta^n_Y))_n$ are negatively convergent to $B^{\theta}$ and $B^{\theta'}$ respectively. This means that for large $n$ we have
    \begin{itemize}
        \item $\D\alpha^n_Y$ is contained in the (positively oriented) interval $(B^\theta,A)\subset\D^-Y$,
        \item $(\sigma_Y^\theta)^{-1}(\D\beta^n_Y)$ is also contained in the (positively oriented) interval $(B^\theta,A)\subset\D^-Y$, and
        \item $(\sigma_Y^\theta)^{-1}(\D\beta^n_Y)$ is disjoint of the (positively oriented) interval $(A,B^{\theta'})\subset\D^-Y$.
    \end{itemize}

    \begin{figure}[h]
\leavevmode \SetLabels
\L(.455*1.02) $A$\\%
\L(.665*.38) $B^\theta$\\%
\L(.245*.365) $B^{\theta'}$\\%
\endSetLabels
\begin{center}
\AffixLabels{\centerline{\includegraphics[width=0.4\textwidth]{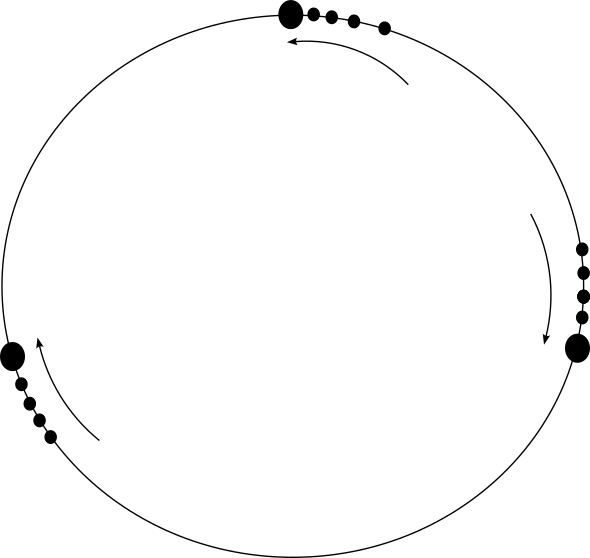}}\hspace{1cm}}
\vspace{-24pt}
\end{center}
\end{figure}

    Since both interval $(B^\theta,A)$ and $(A,B^{\theta'})$ have the same length by \eqref{eq almost done with this}, we deduce that largest distance from points in $(\sigma_Y^\theta)^{-1}(\D\beta^n_Y)$ to points in $\D\alpha^n_Y$ is smaller than the smallest distance from points in $(\sigma_Y^{\theta'})^{-1}(\D\beta^n_Y)$ to points in $\D\alpha^n_Y$, meaning that 
    $$\calD_{\alpha^{n},\beta^{n}}^Y(\theta)\cap\calD_{\alpha^{n},\beta^{n}}^Y(\theta')=\emptyset.$$
    This contradicts \eqref{eq want to be done with this}, showing that $B^\theta= B^{\theta'}$. 

    Now, both points $B^\theta$ and $B^{\theta'}$ are the images of the same $B$ under the gluing maps $(\sigma_Y^\theta)^{-1}$ and $(\sigma_Y^{\theta'})^{-1}$. These two maps differ by the rotation $x\mapsto x+\theta'-\theta$ of $\D^-Y$. Since both of them map $B$ to the same point, it follows that they both agree, meaning that $\theta'-\theta$ is a multiple of $\ell_Y(\D^-Y)$, as we needed to prove. 

    We have proved the proposition when no component of $Y_0$ is a one-holed torus. When a component is a one-holed torus, then we can repeat the argument word by word, replacing $\D^-Y$ by $(\D^-Y)/_{\tiny{\rm anti}}$ and Lemma \ref{lem get arcs1} by Lemma \ref{lem get arcs2}. When we do that we deduce that the equality $B^\theta= B^{\theta'}$ holds in $(\D^-Y)/_{\tiny{\rm anti}}$, from where we deduce that either $B^{\theta'}=B^{\theta}$ or $B^{\theta'}$ is the antipode of $B^\theta$. It follows that $\theta'-\theta$ is a multiple of $\frac 12\ell_Y(\D^-Y)$. This concludes the proof of Proposition \ref{prop wolpert}.
\end{proof}

\begin{appendix}
\section{Appendix}
As we mentioned earlier, Le Quellec \cite{Nolwenn} proved that if $Y_0$ is a compact orientable surface with boundary, then generic points $Y\in\CT(Y_0)$ are determined by their orthospectra $\CG(Y)$. In the proof of Lemma \ref{lem recognize complement up to isometry} we used the fact that this also holds for generic points $Y\in\CT_b(Y_0)$ when $Y_0$ has two boundary components of the same length. The issue is that points in $\CT_b(Y_0)$ are not generic in $\CT(Y_0)$. We explain here how to modify Le Quellec's argument to give us what we want:

\begin{prop}\label{prop appendix}
    Let $Y_0$ be a compact connected surface with 2 boundary components $\D^{\pm}Y_0$. For generic 
    $$Y\in\CT_b(Y_0)=\{Y\in\CT(T_0)\text{ with }\ell_Y(\D^-Y)=\ell_Y(\D^+Y)\}$$ 
    the following holds: if $Y'\in\CT_b(Y_0)$ is such that $\CG(Y')=\CG(Y)$, then $Y$ and $Y'$ are isometric.
\end{prop}

Let us start explaining the basic idea of Le Quellec's argument--it follows closely what Wolpert \cite{Wolpert} and Buser \cite{Buser} did to prove that generic points in Teichm\"uller space are determined by their unmarked length spectrum.

Let $\CA$ be the set of all proper arcs in $Y_0$. The assumption that $\CG(Y')=\CG(Y)$ implies that there is a bijection $\iota:\CA\to\CA$ depending on $Y$ and $Y'$ with 
\begin{equation}\label{eq appendix 1}
\ell_{Y'}(\iota(\tau))=\ell_Y(\tau)\text{ for all }\tau\in\CA.
\end{equation}
Although we are cutting corners here, Le Quellec considers for any such map $\iota:\CA\to\CA$ the set 
$$\CW(\iota)=\{X\in\CT(Y_0)\text{ such that there is }X'\in\CT(Y_0)\text{ satisfying \eqref{eq appendix 1}}\}$$ 
and notices that for any $X\in\CW(\iota)$ the point $X'$ is uniquely determined--this is so because the lengths of the arcs in an hexagon decomposition determine the point in Teichm\"uller space. She then argues that the set $\CW(\iota)$ is analytic, deducing that either it is lower-dimensional or equal to the whole of Teichm\"uller space $\CT(Y_0)$. Moreover, she argues that the union of the sets $\CW(\iota)$ with $\CW(\iota)\neq\CT(Y_0)$ is a countable union of proper analytic sets. Points in there are thus non-generic. It follows that if $Y$ is generic and $Y'$ is such that \eqref{eq appendix 1} holds for some map $\iota:\CA\to\CA$, then $\CW(\iota)=\CT(Y_0)$. When we replace $\CT(Y_0)$ by $\CT_b(Y_0)$, this part of Le Quellec's argument goes through smoothly. 

To conclude the proof of her theorem, Le Quellec argues that if $\CW(\iota)=\CT(Y_0)$ then there is an isometry $F:Y\to Y'$. Accordingly, to prove Proposition \ref{prop appendix} it suffices to prove the following:

\begin{named}{Claim A}
    Suppose that $\iota:\CA\to\CA$ is such that for all $Y\in\CT_b(Y_0)$ there is $Y'\in\CT_b(Y_0)$ such that \eqref{eq appendix 1} holds. Then there is an element $\phi$ in the (extended) mapping class group with $Y'=\phi(Y)$.
\end{named}

In her case, Le Quellec argues this by showing that specific types of arcs are mapped to arcs of the same type. She uses pretty delicate geometric arguments to do this, and unfortunately it would involve quite a bit of work to make this part of her reasoning go through in our setting. To prove the Claim we will rely on much softer geometric arguments, outsourcing most of the work to the fact that every injective simplicial map from the arc complex to itself is induced by a mapping class. Recall that the {\em arc complex} $\BA(Y_0)$ is the simplicial flag complex whose vertices are the homotopy classes of arcs in $Y_0$ and where two vertices span an edge if the two respective homotopy classes have disjoint representatives.

Before launching the proof of Claim A, let us make the following simple observation. 

\begin{lem}\label{lem app 1}
    Suppose that $Y_0$ is not a two-holed torus and let $\tau_1,\tau_2\in\CA$ be disjoint simple arcs. For any $\epsilon>0$ there is $Y\in\CT_b(Y_0)$ with $\ell_Y(\tau_2),\ell_Y(\tau_1)<\epsilon$.
\end{lem}
\begin{proof}
    Let $X_0$ be the surface obtained by identifying (in an orientation reversing way) both components of $Y_0$ and let $\Gamma_0\subset X_0$ be trace of the boundary of $Y_0$. When we cut $Y_0$ along $\tau_1$ and $\tau_2$ we get a surface other than the annulus. A minimal amount of inspection implies that there is a simple multicurve $\gamma\subset X_0$ such that when we cut it along $\Gamma_0$ then, among possibly other arcs, we get $\tau_1,\tau_2$. To get the desired hyperbolic structure, take a metric $X'$ on $X_0$ with $\ell_{X'}(\gamma)<\epsilon$, and then let $Y'$ be the surface obtained by $X'$ cutting along the geodesic representative of $\Gamma_0$.
\end{proof}

We can now prove the claim in the case that $Y_0$ is not a two-holed torus.
    
Let $\CA_s\subset\CA$ be the set of simple arcs and note that the restriction of $\iota$ to $\CA_s$ is injective, just because it is the restriction of a bijection. Lemma \ref{lem app 1} implies that for $\tau\in\CA_s$ there is some $Y\in\CT_b(Y_0)$ where $\ell_Y(\tau)$ is very small. In the surface $Y'$ satisfying \eqref{eq appendix 1} we have then that also $\ell_{Y'}(\iota(\tau))$ is very small. That implies that $\iota(\tau)$ is simple. We have proved that $\iota(\CA_s)\subset\CA_s$. The same argument shows that if $\tau,\tau'\in\CA_s$ are disjoint, then so are $\iota(\tau),\iota(\tau')$. It follows that, thinking of $\CA_s$ as the set of vertices of the arc complex $\BA(Y_0)$, the injective map $\iota:\CA_s\to\CA_s$ extends to a simplicial map
$$\iota:\BA(Y_0)\to\BA(Y_0)$$
from the arc complex to itself. Now the key observation, due to Irmak and McCarthy \cite{Irmak}, is that {\em for any compact, connected, orientable surface with nonempty boundary, any injective map from the arc complex to itself is induced by a homeomorphism of the surface}. It follows that there is $\phi:Y_0\to Y_0$ with
$$\iota(\alpha)=\phi(\alpha)$$
for every simple arc $\alpha\in\CA_s$. Now we get from \eqref{eq appendix 1} that
$$\ell_{\phi^{-1}(Y')}(\tau)=\ell_{Y'}(\phi(\tau))=\ell_Y(\tau)\text{ for all }\tau\in\CA_s.$$
This implies that $\phi^{-1}(Y')=Y\in\CT(Y_0)$, as we needed to prove.

We have proved the claim if $Y_0$ is not a two-holed torus. Let's discuss this remaining case next. The issue here is that if $Y_0$ is a two-holed torus then any $Y\in\CT_b(Y_0)$ has a non-trivial isometry which exchanges both boundary components $\D^-Y$ and $\D^+Y$. Moreover, this isometry has the property that it sends any simple arc $\alpha$ with boundary in $\D^-Y$ to the unique simple arc $\alpha'$ with boundary in $\D^+Y$ which is disjoint of $\alpha$--we refer of $\alpha'$ as the {\em twin} of $\alpha$, and vice versa. This implies that, while for a simple arc $\tau_1$ one can always find $Y\in\CT_b(Y_0)$ with $\ell_Y(\tau_1)<\epsilon$, it might be impossible two get two disjoint simple arcs $\tau_1,\tau_2$ short at the same time. However, what we said so far implies that the map $\iota:\CA\to\CA$ still maps $\CA_s$ into itself, that it maps pairs of twins to pairs of twins, and that it preserves disjointness unless we compare arcs with end-points in the same boundary component of $Y_0$. 

The fact that twins always have the same length means in some sense that there is some indeterminacy for $\iota$. We thus define a new map 
$$\iota':\CA\to\CA\ \iota':\alpha\mapsto\left\{\begin{array}{ll}
    \iota(\alpha)&\text{ if }\alpha\text{ joins both components of }\D Y_0\\
    \iota(\alpha)&\text{ if }\D\alpha\subset\D^\pm Y_0\text{ and }\D\iota(\alpha)\subset\D^{\pm Y_0}\\
    \text{twin of }\iota(\alpha) &\text{ otherwise,}
\end{array}   \right.$$
and note that any two $Y,Y'$ satisfying \ref{eq appendix 1}, do also satisfy $\ell_{Y'}(\iota'(\tau))=\ell_Y(\tau)$ for all $\tau\in\CA$. We claim now that the map $\iota'$ sends disjoint arcs to disjoint arcs. To see that this is the case it suffices to prove that if $\tau_1,\tau_2$ are disjoint simple arcs with $\D\tau_1,\D\tau_2\subset\D^-Y_0$ then $\iota'(\tau_1)$ and $\iota'(\tau_2)$ are also disjoint. To see that this is the case note that if we let $\kappa_1,\kappa_2$ be simple arcs joining $\D^-Y_0$ to $\D^+Y_0$ and disjoint of both $\tau_1$ and $\tau_2$, and we take for $\epsilon>0$ very short any $Y\in\CT_b(Y_0)$ such that 
\begin{equation}\label{eq get metric 2 holed torus}
    \ell_Y(\tau_1),\ \ell_Y(\kappa_1),\text{ and }\ell_Y(\kappa_2)<\epsilon
\end{equation} 
the following holds:
\begin{enumerate}
    \item $\ell_Y(\tau_2)\le -2\log(\epsilon)+10$,
    \item any arc $\alpha$ with $\D\alpha\subset\D^-Y$ with length less that $-4\log(\epsilon)-10$ is disjoint of $\tau_1$.
\end{enumerate}
Evidently the number $10$ could be optimized, but we see no value in doing so. Anyways, the point of this all is that the surface $Y'\in\CT_b(Y_0)$ such that the pair $(Y,Y')$ satisfies \eqref{eq appendix 1} again satisfies \eqref{eq get metric 2 holed torus} when we replace $\tau_1,\kappa_1,\kappa_2$ by $\iota'(\tau_1),\iota'(\kappa_1)$ and $\iota'(\kappa_2)$. Since $\iota'(\tau_2)$ has boundary in $\D^-Y$ (by construction of $\iota'$) and satisfies, for $\epsilon$ small enough, that 
$$\ell_Y'(\iota'(\tau_2))=\ell_Y(\tau_2)\stackrel{\text{(1)}}<-2\log(\epsilon)+10<-4\log(\epsilon)-10$$
we get from (2) that $\iota'(\tau_2)$ is disjoint of $\iota'(\tau_1)$. This proves that $\iota':\CA_s\to\CA_s$ sends disjoint simple arcs to disjoint simple arcs. We can now conclude the proof of the claim exactly as we did when $Y_0$ was not a two-holed torus.\qed

\end{appendix}

\end{document}